\documentclass[10pt,twoside]{amsart}
\usepackage{hyperref}
\usepackage{amsmath} 
\usepackage{verbatim} 
\usepackage{amsthm} 
\usepackage{amsfonts} 
\usepackage{amssymb} 
\usepackage{mathrsfs} 
\usepackage{graphicx} 
\usepackage{multicol}
\usepackage[all]{xy}
\usepackage{pdfsync} 

\numberwithin{equation}{section} 
\theoremstyle{plain}
\newtheorem{theorem}{Theorem}[section]
\newtheorem{proposition}[theorem]{Proposition}
\newtheorem{lemma}[theorem]{Lemma}
\newtheorem{corollary}[theorem]{Corollary}

\theoremstyle{definition}
\newtheorem{definition}[theorem]{Definition}

\theoremstyle{remark}
\newtheorem{remark}[theorem]{Remark}

\newtheorem{questions}[theorem]{Questions}
\newtheorem{example}[theorem]{Example}

\newcommand{\Cs}{{$C^*$-algebra}}
\newcommand{\N}{\mathbb N}
\newcommand{\Z}{\mathbb Z}
\newcommand{\Q}{\mathbb Q}
\newcommand{\R}{\mathbb R}
\newcommand{\C}{\mathbb C}

\newcommand{\E}{\mathbb E}
\newcommand{\F}{\mathbb F}
\newcommand{\K}{\mathbb K}
\newcommand{\cA}{\mathcal A}
\newcommand{\cB}{\mathcal B}
\newcommand{\cC}{\mathcal C}

\newcommand{\cK}{\mathcal K} 
\newcommand{\cI}{\mathcal I} 

\newcommand{\red}{{\mathrm{red}}}
\newcommand{\full}{{\mathrm{full}}}
\newcommand{\ep}{{\varepsilon}}

\newcommand{\Can}{{\bf{K}}}
\newcommand{\se}{\subseteq}

\def\YEAR{\year}\newcount\VOL\VOL=\YEAR\advance\VOL by-1995
\def\firstpage{1597}\def\lastpage{1626}
\def\received{June 3, 2013}\def\revised{July 22, 2013}
\def\communicated{Joachim Cuntz}

\makeatletter
\def\magnification{\afterassignment\m@g\count@}
\def\m@g{\mag=\count@\hsize6.5truein\vsize8.9truein\dimen\footins8truein}
\makeatother

\font\eightrm=cmr8
\font\caps=cmcsc10                    
\font\Caps=cmcsc10 scaled \magstep1   
\font\scaps=cmcsc8

\makeatletter
\@specialpagefalse\headheight=8.5pt
\def\DocMath{{\def\th{\thinspace}\scaps Documenta Math.}}
\renewcommand{\@oddfoot}{\hfill\scaps Documenta Mathematica 
    \number\VOL\  (\number\YEAR) \number\firstpage--\lastpage\hfill}
\renewcommand{\@evenfoot}{\ifnum\thepage>\lastpage\hfill\scaps
    Documenta Mathematica \number\VOL\  (\number\YEAR)\hfill\else\@oddfoot\fi}%

\renewcommand{\@evenhead}{%
    \ifnum\thepage>\lastpage\rlap{\thepage}\hfill%
    \else\rlap{\thepage}\slshape\leftmark\hfill{\caps\SAuthor}\hfill\fi}%
\renewcommand{\@oddhead}{%
    \ifnum\thepage=\firstpage{\DocMath\hfill\llap{\thepage}}%
    \else{\slshape\rightmark}\hfill{\caps\STitle}\hfill\llap{\thepage}\fi}%
\makeatother

\def\TSkip{\bigskip}
\newbox\TheTitle{\obeylines\gdef\GetTitle #1
\ShortTitle  #2
\SubTitle    #3
\Author      #4
\ShortAuthor #5
\EndTitle
{\setbox\TheTitle=\vbox{\baselineskip=20pt\let\par=\cr\obeylines%
\halign{\centerline{\Caps##}\cr\noalign{\medskip}\cr#1\cr}}%
	\copy\TheTitle\TSkip\TSkip%
\def\next{#2}\ifx\next\empty\gdef\STitle{#1}\else\gdef\STitle{#2}\fi%
\def\next{#3}\ifx\next\empty%
    \else\setbox\TheTitle=\vbox{\baselineskip=20pt\let\par=\cr\obeylines%
    \halign{\centerline{\caps##} #3\cr}}\copy\TheTitle\TSkip\TSkip\fi%
\centerline{\caps #4}\TSkip\TSkip%
\def\next{#5}\ifx\next\empty\gdef\SAuthor{#4}\else\gdef\SAuthor{#5}\fi%
\ifx\received\empty\relax
    \else\centerline{\eightrm Received: \received}\fi%
\ifx\revised\empty\TSkip%
    \else\centerline{\eightrm Revised: \revised}\TSkip\fi%
\ifx\communicated\empty\relax
    \else\centerline{\eightrm Communicated by \communicated}\fi\TSkip\TSkip%
\catcode'015=5}}\def\Title{\obeylines\GetTitle}
\def\Abstract{\begingroup\narrower
    \parskip=\medskipamount\parindent=0pt{\caps Abstract. }}
\def\EndAbstract{\par\endgroup\TSkip}

\long\def\MSC#1\EndMSC{\def\arg{#1}\ifx\arg\empty\relax\else
     {\par\narrower\noindent%
     2010 Mathematics Subject Classification: #1\par}\fi}

\long\def\KEY#1\EndKEY{\def\arg{#1}\ifx\arg\empty\relax\else
	{\par\narrower\noindent Keywords and Phrases: #1 \par}\fi\TSkip}

\newbox\TheAdd\def\Addresses{\vfill\copy\TheAdd\vfill
    \ifodd\number\lastpage\vfill\eject\phantom{.}\vfill\eject\fi}
{\obeylines\gdef\GetAddress #1
\Address #2 
\Address #3
\Address #4
\EndAddress
{\def\xs{4.3truecm}\parindent=0pt
\setbox0=\vtop{{\obeylines\hsize=\xs#1\par}}\def\next{#2}
\ifx\next\empty 
     \setbox\TheAdd=\hbox to\hsize{\hfill\copy0\hfill}
\else\setbox1=\vtop{{\obeylines\hsize=\xs#2\par}}\def\next{#3}
\ifx\next\empty 
     \setbox\TheAdd=\hbox to\hsize{\hfill\copy0\hfill\copy1\hfill}
\else\setbox2=\vtop{{\obeylines\hsize=\xs#3\par}}\def\next{#4}
\ifx\next\empty\ 
     \setbox\TheAdd=\vtop{\hbox to\hsize{\hfill\copy0\hfill\copy1\hfill}
                \vskip20pt\hbox to\hsize{\hfill\copy2\hfill}}
\else\setbox3=\vtop{{\obeylines\hsize=\xs#4\par}}
     \setbox\TheAdd=\vtop{\hbox to\hsize{\hfill\copy0\hfill\copy1\hfill}
	        \vskip20pt\hbox to\hsize{\hfill\copy2\hfill\copy3\hfill}}
\fi\fi\fi\catcode'015=5}}\gdef\Address{\obeylines\GetAddress}

\begin{document}
\Title
Non-Supramenable Groups
 Acting on Locally Compact Spaces 
\ShortTitle 
Non-Supramenable Groups
\SubTitle   
\Author 
Julian Kellerhals, Nicolas Monod, and Mikael R\o rdam
\ShortAuthor 
\EndTitle
\Abstract 
Supramenability of groups is characterised in terms of invariant measures on 
locally compact spaces. This opens the door to constructing interesting crossed 
product $C^*$-algebras for non-supramenable groups. In particular, stable Kirchberg 
algebras in the UCT class are constructed using crossed products for both 
amenable and non-amenable groups. 
\EndAbstract
\MSC 
43A07, 46L55, 46L35
\EndMSC
\KEY 
Supramenable groups, actions on locally compact spaces, purely
infinite $C^*$-algebras and actions.
\EndKEY
\Address
Julian Kellerhals
EPFL
SB-EGG
Station 8
1015 Lausanne
Switzerland
julian.kellerhals@epfl.ch
\Address
Nicolas Monod
EPFL
SB-EGG
Station 8
1015 Lausanne
Switzerland
nicolas.monod@epfl.ch
\Address
Mikael R\o rdam
Department of Mathematics
University of Copenhagen
Univer\-si\-tets\-parken~5
DK-2100
Copenhagen \O
Denmark
rordam@math.ku.dk
\Address
\EndAddress

\section{Introduction}
A group $\Gamma$ is called \emph{amenable} if it carries an invariant finitely additive measure $\mu$ with $0<\mu(\Gamma)<\infty$. This concept introduced by von Neumann~\cite{vonNeumann29} was originally motivated by the Hausdorff--Banach--Tarksi paradox~\cite{Hausdorff14_article,Banach-Tarski} and has become central in many aspects of group theory and beyond.

However, from the point of view of paradoxical decompositions, it is more natural to consider the following question of von Neumann~\cite[\S4]{vonNeumann29}. Let $\Gamma$ be a group acting on a set $X$. Given a subset $A\se X$, when is there an invariant finitely additive measure $\mu$ on $X$ with $0<\mu(A)<\infty$?

Following Rosenblatt~\cite{Rosenblatt:TAMS}, the group $\Gamma$ is called \emph{supramenable} if there is such a measure for every non-empty subset; it turns out that it is sufficient to consider the case $X=\Gamma$. Supramenability is much stronger than amenability: whilst it holds for commutative and subexponential groups, it fails already for metabelian groups (Example~\ref{exo:semi}) and thus is not preserved under extensions. It passes to subgroups and quotients but is not known to be preserved under direct products.

One of the many equivalent characterisations of amenability is the existence of a non-zero invariant Radon measure on any compact Hausdorff $\Gamma$-space (see e.g.~\cite[5.4]{Pier}). It turns out that there is an analogous characterisation of supramenability using \emph{locally compact} spaces. We shall establish this and a couple more characterisations:

\begin{theorem}\label{thm:equivalences}
The following conditions are equivalent for any group $\Gamma$.

\begin{enumerate}
\item $\Gamma$ is supramenable.
\item Any co-compact $\Gamma$-action on a locally compact Hausdorff space admits a non-zero invariant Radon measure.\label{it:Radon}
\item The Roe algebra $\ell^\infty(\Gamma) \rtimes_\red \Gamma$ contains no properly infinite projection.\label{it:Roe}
\item There is no injective Lipschitz map from the free group $\F_2$ to $\Gamma$.
\end{enumerate}
\end{theorem}

\noindent
(This also gives a partial answer to the questions raised after Theorem~6.3.1 in~\cite{ShalomQI}. Regarding condition~\ref{it:Roe}, we recall that $\Gamma$ is amenable if and only if the Roe algebra is itself not properly infinite~\cite{RorSie:ETDS}. Theorem~\ref{thm:equivalences} is proved below in Propositions~\ref{prop:lch}, \ref{prop:supra-F_2} and~ \ref{prop:Roe-Supramenable}.)


We shall leverage the \emph{failure} of condition~\ref{it:Radon} in order to provide interesting \Cs s through the corresponding reduced crossed product construction. To this end, we need to establish that the locally compact $\Gamma$-space can be assumed to have several additional properties, as follows. 
Recall that a \emph{Kirchberg algebra} is a \Cs{} which is simple, purely infinite, nuclear and separable. Kirchberg algebras in the UCT class are completely classified by their $K$-theory, see~\cite{Phi:class} and~\cite{Kir:fields}. 

\begin{theorem}\label{thm:d:intro}
Let $\Gamma$ be a countable group. Then $\Gamma$ admits a free, minimal, purely infinite action on  the locally compact non-compact Cantor set $\Can^*$ if and only if $\Gamma$ is non-supramenable. The crossed product \Cs\ $C_0(\Can^*) \rtimes_\red \Gamma$ associated with any such action will be a stable, simple, purely infinite \Cs.

If the non-supramenable group $\Gamma$ is amenable, then the associated crossed product \Cs\ $C_0(\Can^*) \rtimes_\red \Gamma$ will be a stable Kirchberg algebra in the UCT class.
\end{theorem}

In the above, an action is called \emph{purely infinite} if every compact-open subset is paradoxical in a sense made precise in Definition~\ref{def:pi-action} below. It is a strengthening of the failure of condition~\ref{it:Radon} in Theorem~\ref{thm:equivalences}. As for the non-compact Cantor set $\Can^*$, one can realise it e.g.\ as the usual (compact) Cantor set $\Can$ with a point removed, or as $\N\times \Can$.


In the second statement of Theorem~\ref{thm:d:intro}, the amenability of $\Gamma$ is only used to deduce that the action on $\Can^*$ is amenable. But of course there are many amenable actions of non-amenable groups; therefore, we strive to produce more examples of \Cs s as above using crossed products for groups $\Gamma$ that are not necessarily amenable. Here is a summary of some progress in that direction.

\begin{theorem}\label{thm:non-exact:intro}
Let $\Gamma$ be a countable group.
\begin{enumerate}
\item If $\Gamma$ contains an infinite exact subgroup, then $\Gamma$ admits a free minimal amenable action on $\Can$ or on $\Can^*$; necessarily on $\Can^*$ if $\Gamma$ is non-exact. If $\Gamma$ contains an infinite exact subgroup of infinite index, then $\Gamma$ admits a free minimal amenable action on $\Can^*$.
\item If $\Gamma$ contains an element of infinite order, or if $\Gamma$ contains an infinite amenable subgroup of infinite index, then $\Gamma$ admits a free minimal amenable action on $\Can^*$ such that $\Can^*$ admits an invariant non-zero Radon measure. 
The associated crossed product $C_0(\Can^*) \rtimes_\red \Gamma$ is a stably finite simple separable nuclear \Cs{} the UCT class.
\item If $\Gamma$ contains an exact non-supramenable subgroup,  then $\Gamma$ admits a free minimal amenable purely infinite action on $\Can^*$. 
The associated crossed product $C_0(\Can^*) \rtimes_\red \Gamma$ is a stable Kirchberg algebra in the UCT class.
\end{enumerate}
\end{theorem}

\subsection*{Acknowledgements}
Work supported in part by the Swiss National Science Foundation, the European Research Council, Danish National Research Foundation (DNRF) through the Centre for Symmetry and Deformation at University of Copenhagen, and The Danish Council for Independent Research, Natural Sciences.

\section{Actions on locally compact spaces} \label{sec:actions}

\noindent It is well-known that a group is amenable if and only if whenever it acts on a compact Hausdorff space, then the space admits an invariant probability measure. We shall here characterise supramenable groups in a similar way by their actions on locally compact Hausdorff spaces.

Let us first recall the basic notions of comparison of subsets of a group $\Gamma$. If $A, B \subseteq \Gamma$, then write $A \sim_\Gamma B$ if there are finite partitions $\{A_j\}_{j=1}^n$ and $\{B_j\}_{j=1}^n$ of $A$ and $B$, respectively, and elements $t_1, t_2, \dots, t_n \in \Gamma$ such that $A_j = t_jB_j$. Write $A \precsim_\Gamma B$ if $A \sim_\Gamma B_0$ for some $B_0\subseteq B$. A set $A \subseteq \Gamma$ is said to be \emph{paradoxical} if there are disjoint subsets $A_0,A_1$ of $A$ such that $A \sim_\Gamma A_0 \sim_\Gamma A_1$. Tarski's theorem says that $A\subseteq \Gamma$ is non-paradoxical if and only if there is an invariant finitely additive measure $\mu$ defined  on the entire power set, $P(\Gamma)$, of $\Gamma$, such that $\mu(A) = 1$. (Notice that the empty set is paradoxical.)

The following proposition is well-known. For the weaker case of functions with ``$A$-bounded'' support, see for example~\cite[Chap.~1]{Greenleaf}.

\begin{proposition} \label{prop:a}
Let $\Gamma$ be a group and let $\mu$ be a finitely additive measure on $\Gamma$. Let $V_\mu$ be the subspace of $\ell^\infty(\Gamma)$ consisting of all $f \in \ell^\infty(\Gamma)$ such that $\mu(\mathrm{supp}(f)) < \infty$. It follows that there is a unique positive linear functional $I_\mu \colon V_\mu \to \C$ such that $I_\mu(1_E) = \mu(E)$ for all $E \subseteq \Gamma$ with $\mu(E) < \infty$. If $\mu$ is $\Gamma$-invariant, then so is $I_\mu$.
\end{proposition}

\begin{proof}
Let $\mathcal{F}$ be the collection of subsets of $\Gamma$ of finite measure. For each $F \in \mathcal{F}$ let $V_F$ be the subspace of  $\ell^\infty(\Gamma)$ consisting of all functions with support in $F$. The family $\{V_F\}$ is upwards directed with union $V_\mu$. Using that the set of bounded functions $F \to \C$ that take finitely many values is uniformly dense in $V_F$ it is easy to see that there is a unique positive (necessarily bounded) linear functional $I_\mu^{(F)} \colon V_F \to \C$  such that $I_\mu^{(F)} (1_E) = \mu(E)$ for all $E \subseteq F$. It follows that there exists a unique positive (unbounded) linear functional $I_\mu \colon V_\mu \to \C$ that extends all the $I_\mu^{(F)}$'s. 

When $\mu$ is translation invariant, the invariance of $I_\mu$ follows from uniqueness of $I_\mu$. 
\end{proof}

\noindent Let $\Gamma$ be a group which acts on a locally compact Hausdorff space $X$. We say that the action is \emph{co-compact} if there is a compact subset $K$ of $X$ such that $\bigcup_{t \in \Gamma} t.K = X$. Any \emph{minimal} action of any group on any locally compact Hausdorff space is automatically co-compact.

Observe that if $K$ is as above, and if $\lambda$ is a non-zero Radon measure on $X$, then $0 < \lambda(K) < \infty$. 

If $X$ is locally compact and Hausdorff and if $\Gamma$ acts co-compactly on $X$, then there is a compact subset $K \subseteq X$ such that $\bigcup_{t \in \Gamma} t. K^{\mathrm{o}} = X$, where  $K^{\mathrm{o}}$ denotes the interior of $K$. This follows easily from the fact that for each compact set $K' \subseteq X$ there is another compact set $K \subseteq X$ such that $K' \subseteq K^{\mathrm{o}}$.

\begin{lemma} \label{lm:lch}
Let $\Gamma$ be a group acting co-compactly on a locally compact Hausdorff space $X$. For each compact subset $K$ of $X$, such that $\bigcup_{t \in \Gamma} t.K^\mathrm{o} = X$, and for each $x_0 \in X$, put
$$A(K,x_0) = \{t \in \Gamma \mid t.x_0 \in K\}.$$
If $A(K,x_0)$ is non-paradoxical in $\Gamma$ for some $K$ and $x_0$ as above, then there is a non-zero $\Gamma$-invariant Radon measure $\lambda$ on $X$.
\end{lemma}

\begin{proof}
By Tarski's theorem there exists an invariant finitely additive measure $\mu$ on $\Gamma$ such that $\mu(A(K,x_0)) = 1$. 

Each compact set $L \subseteq X$ is contained in $\bigcup_{s \in F} s.K$ for some finite subset $F$ of $\Gamma$. Hence
\begin{equation} \label{eq:aa}
\mu\big(\{t \in \Gamma \mid t.x_0 \in L \}\big) \le \mu\big(\bigcup_{s \in F} sA(K,x_0)\big) \le |F| \mu(A(K,x_0)) < \infty.
\end{equation}
For each $f \in C_c(X)$ let $\hat{f} \in \ell^\infty(\Gamma)$ be given by $\hat{f}(t) = f(t.x_0)$, $t \in \Gamma$. In the notation of Proposition~\ref{prop:a}, it follows from \eqref{eq:aa} that $\hat{f} \in V_\mu$. Let $I_\mu \colon V_\mu \to \C$ be the $\Gamma$-invariant functional associated with $\mu$ constructed in  Proposition~\ref{prop:a}. Define $\Lambda \colon C_c(X) \to \C$ by
$\Lambda(f) = I_\mu(\hat{f})$. As $I_\mu$ is $\Gamma$-invariant, so is $\Lambda$. By Riesz' representation theorem there is a Radon measure $\lambda$ on $X$ such that $\Lambda(f) = \int_X f \, d\lambda$. As $\Lambda$ is $\Gamma$-invariant, so is $\lambda$. 

If $f \in C_c(X)$ is such that  $f \ge 1_K$, then $\hat{f} \ge 1_{A(K,x_0)}$, so $\Lambda(f) = I_\mu(\hat{f}) \ge \mu(A(K,x_0)) = 1$. This shows that $\lambda(K) \ge 1$, so $\lambda$ is non-zero.
\end{proof}

\noindent The action of a group $\Gamma$ on itself given by left-multiplication extends to an action of $\Gamma$ on its beta-compactification, $\beta \Gamma$. We shall refer to this action as the \emph{canonical action of $\Gamma$ on $\beta \Gamma$}.

\begin{definition} \label{def:KA-XA}
Fix a subset $A$ of $\Gamma$. Denote by $K_A$ its closure in $\beta \Gamma$, and put
$$X_A = \bigcup_{t \in \Gamma} t.K_A.$$
\end{definition}

\noindent 
As $\beta \Gamma$ is Stonian (a.k.a.\ extremely disconnected), it follows that $K_A$ is compact-open in $\beta \Gamma$.  Moreover, $X_A$ is an open $\Gamma$-invariant subset of $\beta \Gamma$. In particular, $X_A$ is locally compact and Hausdorff, and $\Gamma$ acts co-compactly on $X_A$. Note also that $\Gamma \subseteq X_A$ (if $A \ne \emptyset$), and that $\Gamma = X_A$ if and only if $A$ is finite and non-empty.

We state below two easy lemmas that shall be used frequently in what follows.

\begin{lemma} \label{lem:KA} Let $\Gamma$ be a group.
\begin{enumerate}
\item Let $A$ and $B$ be subsets of $\Gamma$. If $A \cap B = \emptyset$, then $K_A \cap K_B = \emptyset$.
\item For any subset $A$ of $\Gamma$ we have $A = K_A \cap \Gamma$.
\item If $A$ and $B$ are subsets of $\Gamma$, then $A \subseteq B$ if and only if $K_A \subseteq K_B$.
\item If $K$ is a compact-open subset of $\beta \Gamma$ and if $A = K \cap \Gamma$, then $K=K_A$.
\item If $A$ is a subset of $\Gamma$ and $t \in \Gamma$, then $t.K_A = K_{tA}$.
\item $\beta \Gamma$ and $\emptyset$ are the only $\Gamma$-invariant compact-open subset of $\beta \Gamma$.
\end{enumerate}
\end{lemma}

\begin{proof} (i). As $A$ is an open subset of $\beta \Gamma$ we have $K_B \subseteq \beta \Gamma \setminus A$. Hence $A \subseteq \beta \Gamma \setminus K_B$, which entails that $K_A \subseteq \beta \Gamma \setminus K_B$, because $K_B$ is open.

(ii). It is clear that $A \subseteq K_A \cap \Gamma$. Conversely, if $t \in K_A \cap \Gamma$, then $\{t\}$ is an open subset of $\beta \Gamma$ that intersects $K_A$, so it also intersects $A$, i.e., $t \in A$.

(iii). The non-trivial implication follows from part (ii).

(iv). It is clear that $K_A \subseteq K$. If $x \in K$ and if $V$ is an open neighbourhood of $x$, then $K \cap V$ is a non-empty open subset of $\beta \Gamma$. As $\Gamma$ is an open and dense subset of $\beta \Gamma$ we deduce that $A \cap V = \Gamma \cap K \cap V$ is non-empty. As $V$ was an arbitrary open neighbourhood of $x$, we conclude that $x \in K_A$.

(v). As $A \subseteq K_A$ we have $tA \subseteq t.K_A$, so $K_{tA} \subseteq t.K_A$. Applying this inclusion with $t^{-1}$ in the place of $t$ we get
$$t.K_A = t.K_{t^{-1}t.A} \subseteq t.\big(t^{-1}.K_{tA}\big) = K_{tA}.$$

(vi). This follows from (iv) and (v) and the fact that $\Gamma$ and $\emptyset$ are the only $\Gamma$-invariant subsets of $\Gamma$.
\end{proof}

\noindent The lemma above says that there is a one-to-one correspondence between subsets of $\Gamma$ and compact-open subsets of $\beta \Gamma$. This is not surprising. Indeed the function algebras $\ell^\infty(\Gamma)$ and $C(\beta \Gamma)$ are canonically isomorphic (by the definition of the beta-compactification!). The canonical isomorphism carries the projection $1_A \in \ell^\infty(\Gamma)$ onto $1_{K_A} \in C(\beta \Gamma)$. We shall work with both pictures in this paper.

If $A$ and $B$ are subsets of a group $\Gamma$, then write $A \propto_\Gamma B$ if $A$ is \emph{$B$-bounded} i.e., if $A \subseteq \bigcup_{t \in F} tB$ for some finite subset $F$ of $\Gamma$.

\begin{lemma} \label{lm:A-B-comp} Let $A$ and $B$ be subsets of a group $\Gamma$.
\begin{enumerate}
\item $K_B \subseteq X_A$ if and only if $B \propto_\Gamma A$.
\item $X_A = X_B$ if and only if $A \propto_\Gamma B \propto_\Gamma A$.
\item If $A \propto_\Gamma B \propto_\Gamma A$, then $A$ is paradoxical if and only if $B$ is paradoxical.
\end{enumerate}
\end{lemma}

\begin{proof}
(i). If $B \subseteq \bigcup_{t \in F} tA$ for some finite subset $F$ of $\Gamma$, then $$K_B \subseteq  K_{\bigcup_{t \in F} tA} = \bigcup_{t \in F} t.K_A \subseteq X_A$$ by Lemma~\ref{lem:KA}(iii) and (v).  Conversely, if $K_B$ is a subset of $X_A = \bigcup_{t \in \Gamma} t.K_A$, then $K_B \subseteq \bigcup_{t \in F} t.K_A = K_{\bigcup_{t \in F} tA}$ for some finite subset $F$ of $\Gamma$ by compactness of $K_B$. Hence $B$ is contained in $\bigcup_{t \in F} tA$ by Lemma~\ref{lem:KA}(iii).

(ii). Note that $K_B \subseteq X_A$ if and only if $X_B \subseteq X_A$ because $X_A$ is $\Gamma$-invariant. Thus (ii) follows from (i).

(iii). Suppose that  $B \subseteq \bigcup_{t \in F} tA$ and $A \subseteq \bigcup_{t \in F'} tB$ for some finite subsets $F$ and $F'$ of $\Gamma$. Suppose that $A$ is non-paradoxical. Then, by Tarski's theorem, we can find a finitely additive invariant measure $\lambda$ on $\Gamma$ such that $\lambda(A)=1$. Since $\lambda(B) \le |F|\lambda(A) $ and $\lambda(A) \le |F'|\lambda(B)$, we conclude that $0 < \lambda(B) < \infty$, which in turns implies that $B$ is non-paradoxical.
\end{proof}

\noindent It follows from Lemma~\ref{lem:KA}(vi) and from Lemma~\ref{lm:A-B-comp} that if $A$ is a non-empty subset of $\Gamma$, then $X_A$ is compact if and only if $X_A = \beta \Gamma$ if and only if $\Gamma \propto A$.

\begin{proposition} \label{prop:XA}
Let $A$ be a non-empty subset of $\Gamma$. There is a non-zero $\Gamma$-invariant Radon measure on $X_A$ if and only if $A$ is non-paradoxical in $\Gamma$.
\end{proposition}

\begin{proof}
In the notation of Lemma~\ref{lm:lch} it follows from Lemma~\ref{lem:KA}(ii) that $A = A(K_A,e)$, and hence that $X_A$ admits a non-zero $\Gamma$-invariant Radon measure if $A$ is non-paradoxical.

Suppose next that  there is a non-zero $\Gamma$-invariant Radon measure $\lambda$ on $X_A$. Then, necessarily, $0 < \lambda(K_A) < \infty$. Let $\Omega_A$ be the set of all $A$-bounded subsets of $\Gamma$, i.e., $E \in \Omega_A$ if and only if $E \propto_\Gamma A$ if and only if $K_E \subseteq X_A$, cf.\ Lemma~\ref{lm:A-B-comp}(i). 
Let $\mu$ be the invariant finitely additive measure on $\Gamma$ defined by 
$$\mu(E) = \begin{cases} \lambda(K_E), & E \in \Omega_A, \\
\infty, & E \notin \Omega_A. \end{cases}$$
Use Lemma~\ref{lem:KA}(i) to see that $\mu$, indeed, is finitely additive. As $\mu(A) = \lambda(K_A)$ by definition we conclude that $A$ is non-paradoxical.
\end{proof}

\begin{proposition} \label{prop:lch}
A group $\Gamma$ is supramenable if and only if whenever it acts co-compactly on a locally compact Hausdorff space $X$, then $X$ admits a non-zero $\Gamma$-invariant Radon measure. 
\end{proposition}

\begin{proof} 
If $\Gamma$ is not supramenable and if $A$ is a non-empty paradoxical subset of $\Gamma$, then $\Gamma \curvearrowright X_A$ is an example of a co-compact action on a locally compact Hausdorff space that admits no non-zero invariant Radon measure, cf.\ Proposition~\ref{prop:XA}.

The reverse implication follows from Lemma~\ref{lm:lch} (and the remark above that lemma).
\end{proof}

\noindent Proposition~\ref{prop:lch} says that every non-supramenable group admits a co-compact action on a locally compact Hausdorff space for which there is no non-zero invariant Radon measures. Towards proving Theorem~\ref{thm:d:intro}, we shall later improve this result and show that one always can choose this space to be the locally compact non-compact Cantor set, and that one further can choose the action to be free, minimal and \emph{purely infinite} (see Definition~\ref{def:pi-action}).

\begin{example}\label{exo:semi}
Let $\Gamma$ denote the $ax+b$ group with $a \in \Q^+$ and $b \in \Q$. It is well-known that $\Gamma$ is non-supramenable. In fact, $\Gamma$ contains a free semigroup of two generators. One can take the two generators to be $2x$ and $2x+1$. 

It also follows from Proposition~\ref{prop:lch} that $\Gamma$ is non-supramenable, as the canonical action of $\Gamma$ on $\R$ admits no non-zero invariant Radon measure. 
 We can use  Lemma~\ref{lm:lch}  to construct paradoxical subsets of $\Gamma$ other than the ones coming from free sub-semigroups. Indeed, applying Lemma~\ref{lm:lch} to the compact set $K = [\alpha,\beta] \subseteq \R$ and to $x_0 = \gamma \in \R$ we find that
$$\{ax+b \mid \alpha \le a \gamma + b \le \beta\} \subseteq \Gamma$$
is paradoxical whenever $\alpha, \beta, \gamma \in \R$ and $\alpha < \beta$. 
\end{example}

\section{A geometric description of supramenable groups}

\noindent We shall here describe supramenability in terms of the existence of injective quasi-isometric embeddings of free groups. 

\begin{definition} \label{def:Lipschitz}
Let $\Gamma$ and $\Lambda$ be groups. A map $f \colon \Gamma \to \Lambda$ is said to be \emph{Lipschitz} if for every finite set $S \subseteq \Gamma$ there is a finite set $T \subseteq \Lambda$ such that
\begin{equation} \label{eq:L1}
\forall x,y \in \Gamma: xy^{-1} \in S \implies f(x)f(y)^{-1} \in T.
\end{equation}
A map $f \colon \Gamma \to \Lambda$ is said to be a \emph{quasi-isometric embedding} if it is Lipschitz and satisfies that for every finite set $T \subseteq \Lambda$ there is a finite set $S \subseteq \Gamma$ such that
\begin{equation} \label{eq:L2}
\forall x,y \in \Gamma: f(x)f(y)^{-1} \in T \implies xy^{-1} \in S.
\end{equation}
\end{definition}
This definition corresponds to the usual terminology if $\Gamma$ and $\Lambda$ are considered as metric spaces with a right-invariant word metric. If a map $f$ has either of the two properties, the map $\widehat{f}$ defined by $\widehat{f}(x) = f(x^{-1})^{-1}$ satisfies its left-invariant analogue and vice versa. 

If \eqref{eq:L1} holds for some $S \subseteq \Gamma$ and some $T \subseteq \Lambda$, then, for each integer $n \ge 1$, \eqref{eq:L1} also holds with $S^n$ and $T^n$ in the place of $S$ and $T$, respectively. In particular, if $\Gamma$ is finitely generated, then $f \colon \Gamma \to \Lambda$ is Lipschitz if \eqref{eq:L1} holds for some finite symmetric generating set $S$ for $\Gamma$ and some finite set $T \subseteq \Lambda$. 

Every group homomorphism is Lipschitz, and every group homomorphism with finite kernel is a quasi-isometric embedding.

Let $\Gamma$ be a group and let $A \subseteq \Gamma$. A map $\sigma \colon A \to \Gamma$ is said to be a \emph{piecewise translation} if it is injective and if there is a finite set $S \subseteq \Gamma$ such that $\sigma(x)x^{-1} \in S$ for all $x \in A$.  The composition of piecewise translations (when defined) is again a piecewise translation.

If $\sigma \colon A \to \Gamma$ is a piecewise translation, then $\sigma(A_0) \sim_\Gamma A_0$ for all $A_0 \subseteq A$. A non-empty subset $A \subseteq \Gamma$ is paradoxical if and only if there are piecewise translations $\sigma^\pm \colon A \to A$ with disjoint images. In this case, for each natural number $n$, there are piecewise translations $\sigma_j \colon A \to A$, $j = 1,2 , \dots, n$, with disjoint images. 

\begin{lemma} \label{lm:piecewise_translation}
Let $\Gamma$ and $\Lambda$ be groups and let $f \colon \Gamma \to \Lambda$ be an injective Lipschitz map. Let $A$ be a subset of $\Gamma$ and let $\sigma \colon A \to \Gamma$ be a piecewise translation. It follows that the map $\tau \colon f(A) \to \Lambda$, given by $\tau \circ f = f \circ \sigma$, is a piecewise translation.
\end{lemma}

\begin{proof} Injectivity (and well-definedness) of $\tau$ follows from injectivity of $f$ and of $\sigma$.

There is a finite set $S \subseteq \Gamma$ such that $\sigma(x)x^{-1} \in S$ for all $x \in A$. As $f$ is Lipschitz there is a finite set $T \subseteq \Lambda$ such that $f(x)f(y)^{-1} \in T$ whenever $xy^{-1} \in S$.  Hence 
$$\tau(f(x)) f(x)^{-1} = f(\sigma(x))f(x)^{-1} \in T$$
for all $x \in A$, which proves that $\tau$ is a piecewise translation.
\end{proof}

\begin{lemma} \label{lm:Lip-paradoxical}
Let $\Gamma$ and $\Lambda$ be groups and let $f \colon \Gamma \to \Lambda$ be an injective Lipschitz map. Then $f(A)$ is a paradoxical subset of $\Lambda$ whenever $A$ is a paradoxical subset of $\Gamma$. 
\end{lemma}

\begin{proof} Assume that $A$ is a non-empty paradoxical subset of $\Gamma$. Then there are piecewise translations $\sigma^\pm \colon A \to A$ with disjoint images.  Put $B = f(A)$, and let $\tau^\pm \colon B \to B$ be given by $\tau^\pm \circ f = f \circ \sigma^\pm$. Then $\tau^\pm$ are piecewise translations (by Lemma~\ref{lm:piecewise_translation}) and $\tau^+(B) \cap \tau^-(B) = \emptyset$ (by injectivity of $f$). This shows that $B$ is paradoxical.
\end{proof}

\begin{proposition}\label{prop:supra-F_2}
A group $\Gamma$ is supramenable if and only if there is no injective Lipschitz map $f \colon \F_2 \to \Gamma$.
\end{proposition}

\begin{proof} The ``only if'' part follows from Lemma~\ref{lm:Lip-paradoxical} and the fact that $\F_2$ is non-amenable and hence paradoxical. 

Suppose that $\Gamma$ is not supramenable and let $A$ be a non-empty paradoxical subset of $\Gamma$. Then we can find four piecewise translations $\sigma^\pm, \tau^\pm \colon A \to A$ with disjoint images such that there exists $s_0 \in A$ not in the image of any of these four piecewise translations. 

Let $a,b$ denote the generators of $\F_2$ and define a map $f  \colon \F_2 \to A$ by induction on the word length $\ell$ as follows. We first take care of length zero by setting $f(e) = s_0$.  If $x \in \F_2$ has length $\ell(x)\ge 1$, and if $x = a^{\pm 1}x'  \in  \F_2$ is a reduced word, then put $f(x) = \sigma^\pm \circ f(x')$. Similarly, if $x = b^{\pm 1}x' \in  \F_2$ is a reduced word, then let $f(x) =  \tau^\pm \circ f(x')$. 

We prove by contradiction that $f$ is injective. Indeed, if not, we can choose two reduced words  $x\neq y$ such that $f(x) = f(y)$ and with $\ell (x)$ minimal for these conditions. Because $\sigma^\pm$ and $\tau^\pm$ have disjoint images (not containing $s_0$), there exists $c\in\{a,a^{-1},b,b^{-1}\}$ such that $x=cx'$ and $y=cy'$ are reduced words. But then $f(x') = f(y')$, contradicting the minimality of $\ell(x)$. Therefore $f$ is injective.

Let us finally show that $f$ is Lipschitz. By the remarks below Definition~\ref{def:Lipschitz} we need only show that there is a finite set $T \subseteq \Gamma$ such that \eqref{eq:L1} holds with $S =  \{a,a^{-1},b,b^{-1}\}$. This is done by first choosing
$$T' = \left\{\sigma^\pm(x)x^{-1} \mid x \in A\right\}\cup\left\{\tau^\pm(x)x^{-1} \mid x \in A\right\}.$$
If $x,y \in \F_2$ and $xy^{-1} \in S$, then either $x = cy$ or $y = cx$ for some $c \in S$ and such that $cy$, respectively, $cx$, is reduced. In the former case $f(x) = \sigma^\pm\circ f(y)$ or $f(x) = \tau^\pm\circ f(y)$, hence $f(x)f(y)^{-1} \in T'$. In the latter case $f(y)f(x)^{-1}\in T'$. Finally we can choose $T=T'\cup T'^{-1}$.
\end{proof}

\noindent
Benjamini and Schramm proved in~\cite{BenSchramm:GAFA} that there is an injective quasi-isometric embedding from $\F_2$ into any non-amenable group $\Gamma$. This gives us the following sequence of inclusions:
\begin{eqnarray*}
& &\{\text{groups of sub-exponential growth}\} \\
&\subseteq & \{\text{supramenable groups}\} \\
&=& \{\text{groups with no injective Lipschitz inclusions of} \; \mathbb{F}_2\} \\
&\subseteq & \{\text{groups with no quasi-isometric embedding of} \; \mathbb{F}_2\} \\
&\subset & \{\text{amenable groups}\}
\end{eqnarray*}
De Cornulier and Tessera showed in~\cite[5.1]{Cornulier-Tessera_quasi} that any finitely generated solvable group with exponential growth contains a quasi-isometrically embedded free sub-semigroup on two generators. Solvable groups are amenable and the free group $\F_2$ embeds quasi-isometrically in the free semi-group on two generators. Therefore the last inclusion is strict and for solvable groups the first four sets are equal.

We conclude this section by using the theorem of Benjamini and Schramm mentioned above to show that each non-amenable group contains \emph{small} non-empty paradoxical subsets. 

\begin{definition} \label{def:small}
Let $A$ be a subset of a group $\Gamma$. 
\begin{enumerate}
\item Write $A \ll \Gamma$ if $A \precsim_\Gamma \Gamma \setminus B$ for all $A$-bounded subsets $B$ of $\Gamma$.
\item Say that $A$ is \emph{absorbing} if $\bigcap_{t \in S} tA \ne \emptyset$ for all finite sets $S\subseteq \Gamma$.
\item Write $A \ll^* \Gamma$ if all $A$-bounded subsets of $\Gamma$ are non-absorbing.
\end{enumerate}
\end{definition}

\noindent
For example, if $A$ is a subgroup of $\Gamma$, then $A \ll \Gamma$ if and only if $|\Gamma :A| = \infty$. A subset of $\Gamma = \Z$ is absorbing if and only if it contains arbitrarily long consecutive sequences (compare~\ref{pt:absorbing} below for the non-commutative case).

We list some elementary properties of the relations defined above.

\begin{lemma} \label{lem:small}
Let $\Gamma$ be a group and let $A,B$ be subsets of $\Gamma$.
\begin{enumerate}
\item If $B \propto A$ and $A \ll \Gamma$, then $B \ll \Gamma$.
\item If $A \ll \Gamma$, then  $A \ll^* \Gamma$.
\item $A$ is absorbing if and only if for every finite subset $F\subseteq \Gamma$ there is $g\in \Gamma$ with $F g\subseteq A$.\label{pt:absorbing}
\item Suppose $A$ absorbing. Then $\Gamma$ is amenable if and only if there is a left-invariant mean on $\Gamma$ supported on $A$.
\end{enumerate}
\end{lemma}

\begin{proof} (i). As $B \propto A$, i.e., $B$ is $A$-bounded, we have $B \subseteq \bigcup_{s \in S} sA$ for some finite $S \subseteq \Gamma$. Let $n$ be the number of elements in $S$.

Let $C$ be a $B$-bounded subset of $\Gamma$. Then $C$ is also $A$-bounded.
By the hypothesis on $A$ there exists a piecewise translation $\sigma_1 \colon A \to \Gamma \setminus C$. Since $(\sigma_1(A) \cup C) \propto A$ we can repeat the process and obtain piecewise translations $\sigma_1, \dots, \sigma_n \colon A \to \Gamma \setminus C$ with disjoint images. These piecewise translations can be assembled to one piecewise translation $\sigma \colon \bigcup_{s \in S} sA \to \Gamma \setminus C$. The restriction of $\sigma$ to $B$ will then witness the relation $B \precsim_\Gamma \Gamma \setminus C$.

(ii). Note first that if $E \subseteq \Gamma$ is such that $E \ll \Gamma$, then $E$ cannot be absorbing. Indeed, there is a piecewise translation $\sigma \colon E \to \Gamma \setminus E$, and $\sigma(E) \subseteq \bigcup_{s \in S} sE$ for some finite $S \subseteq \Gamma$. Put $T = S \cup \{e\}$. Then $\bigcap_{t \in T} tE = \emptyset$. 

Suppose that $A \ll^* \Gamma$ does not hold. Then there exists an absorbing set $B \subseteq \Gamma$ such that $B \propto A$. Hence $B \ll \Gamma$ does not hold, which by (i) implies that $A \ll \Gamma$ does not hold.

(iii). It suffices to observe that for any $F\subseteq \Gamma$ we have
$$\{g\in\Gamma : Fg\subseteq A\} \ =\ \bigcap_{t\in F^{-1}} t A.$$

(iv). This follows from the fact that left F\o lner sets can be arbitrarily translated by right multiplication.
\end{proof}

\noindent One can show that $A \ll^* \Gamma$ does not imply $A \ll \Gamma$ (there are counterexamples with $\Gamma = \Z$).  We now turn to the existence of interesting small sets in groups.

\begin{lemma} \label{lm:smallA}
Let $\Gamma$ and $\Lambda$ be groups and let $f \colon \Gamma \to \Lambda$ be an injective quasi-isometric embedding. Then $f(A) \ll \Lambda$ whenever $A \ll \Gamma$.

In particular, each non-amenable group $\Gamma$ contains a non-empty paradoxical subset $A$ with $A \ll \Gamma$. 
\end{lemma}

\begin{proof} Suppose that $A \ll \Gamma$. Let $T$ be a finite subset of $\Lambda$ and find a finite set $S \subseteq \Gamma$ such that 
\eqref{eq:L2} holds. There is a piecewise translation $\sigma \colon A \to \Gamma \setminus \bigcup_{s \in S} sA$. This implies that $\sigma(x)y^{-1} \notin S$ for all $x,y \in A$. Let $\tau \colon f(A) \to \Lambda$  be defined by $\tau \circ f = f \circ \sigma$. Then $\tau$ is a piecewise translation by Lemma~\ref{lm:piecewise_translation}. Moreover,
$$\tau(f(x))f(y)^{-1} = f(\sigma(x)) f(y)^{-1} \notin T$$
for all $x,y \in A$. Hence $\tau$ maps $f(A)$ into $\Lambda \setminus \bigcup_{t \in T} t f(A)$. This proves that $f(A) \ll \Lambda$.

If $\Gamma$ is non-amenable then, as noted above, by the theorem of Benjamini and Schramm in ~\cite{BenSchramm:GAFA} there is an injective quasi-isometric embedding $f \colon \F_2 \to \Gamma$. The free group $\F_2$ contains (several)  non-amenable subgroups $H$ of infinite index. For any such subgroup $H$ we can put $A = f(H)$. Then $A \ll \Gamma$ by the first part of the lemma, and $A$ is paradoxical by Lemma~\ref{lm:Lip-paradoxical}.
\end{proof}

\noindent In general, if we are not asking that $A$ in the lemma above is paradoxical, we have the following:

\begin{lemma} \label{lm:smallA2}
Each infinite group $\Gamma$ contains an infinite subset $A$ such that $A \ll \Gamma$.
\end{lemma}

\begin{proof} Choose a sequence $x_1,x_2, x_3, \dots$ of elements in $\Gamma$ such that
\begin{equation} \label{eq:seq}
x_n \notin \{x_k x_\ell^{-1} x_m \mid 1 \le k,\ell,m < n\}
\end{equation}
for all $n \ge 2$, and put $A = \{x_1,x_2,x_3, \dots \}$. We claim that $|sA \cap A| \le 2$ for all $s \in \Gamma \setminus \{e\}$. Indeed, if $sA \cap A \ne \emptyset$ and $s \ne e$, then $sx_n = x_m$ for some $n \ne m$. Suppose that $k, \ell \in \N$ is another pair such that $sx_k = x_\ell$. Suppose also that $k > \ell$. Then $x_k = s^{-1}x_\ell = x_nx_m^{-1}x_\ell$, whence $k \le \max\{n,m,\ell\}$ by \eqref{eq:seq}, i.e. $k \le \max\{n,m\}$. Similarly, if $\ell > k$, then $x_\ell = x_mx_n^{-1}x_k$, which entails that $\ell \le \max\{n,m\}$. Now, interchanging the roles of the pairs $(n,m)$ and $(k,\ell)$, we conclude that $\max\{n,m\} = \max\{k,\ell\}$ whenever $sx_n = x_m$ and $sx_k = x_\ell$. It follows easily from this observation that $|sA \cap A| \le 2$.

Let $F \subset \Gamma$ be finite, choose $s \in \Gamma \setminus F$, and put
$$H = A \cap \bigcup_{t \in F} s^{-1} tA = \bigcup_{t \in F} (A \cap s^{-1}tA).$$ Then $H$ is finite and
\begin{equation} \label{eq:A-H}
A \setminus H \sim_\Gamma s(A \setminus H) \subseteq \Gamma \setminus \bigcup_{t \in F} tA.
\end{equation}
To  show that  $A \precsim_\Gamma \Gamma \setminus \bigcup_{t \in F} tA$ it now suffices to show that 
$$H \precsim_\Gamma \Gamma \setminus \big(\bigcup_{t \in F} tA \cup sA\big) =  \Gamma \setminus \bigcup_{t \in F'} tA,$$
where $F' = F \cup \{s\}$. As $H$ is finite it suffices to show that $\big| \Gamma \setminus \bigcup_{t \in F'} tA\big| \ge |H|$. However, one can deduce from  \eqref{eq:A-H} that $\Gamma \setminus \bigcup_{t \in F} tA$ is infinite for \emph{any} finite set $F' \subset \Gamma$; and this completes the proof.
\end{proof}

\section{Structure of crossed product \Cs s} \label{sec:C*}

\noindent We review here some, mostly well-known, results about the structure of crossed product \Cs s; and we introduce the notion of \emph{purely infinite actions}. Recall first that a dynamical system $\Gamma \curvearrowright X$ is said to be \emph{regular} if the full and the reduced crossed product \Cs s coincide, i.e., if the natural epimorphism 
$$C_0(X) \rtimes_\full \Gamma \to C_0(X) \rtimes_\red \Gamma$$
is an isomorphism.
Anantharaman-Delaroche proved in~\cite{Ana:TAMS2002} that $\Gamma \curvearrowright X$ is regular if the action is \emph{amenable} (in the sense of Anantharaman-Delaroche,~\cite[Definition 2.1]{Ana:TAMS2002}). Matsumura proved later that the converse also holds when $X$ is compact and the group $\Gamma$ is exact,~\cite{Mat:regular}.

 Moreover, the action is amenable  if and only if the crossed product \Cs{} $C_0(X) \rtimes_\red \Gamma$ is nuclear. Archbold and Spielberg proved in~\cite{ArcSpi:free} that the full crossed product $C_0(X) \rtimes_\full \Gamma$ is simple if and only if the action $\Gamma \curvearrowright X$ is minimal, topologically free and regular. Combining this result of Archbold and Spielberg with the results of Anantharaman-Delaroche mentioned above, we get:

\begin{proposition}[Anantharaman-Delaroche, Archbold--Spielberg] \label{prop:simple}
Let $\Gamma$ be a countable group acting on a locally compact Hausdorff space $X$. Then $C_0(X) \rtimes_\red \Gamma$ is simple and nuclear if and only if  $\Gamma \curvearrowright X$ is minimal, topologically free and amenable.
\end{proposition}

\noindent Recall that a \emph{Kirchberg algebra} is a \Cs{} which is simple, purely infinite, nuclear and separable. 

The proposition below is essentially contained in~\cite[Theorem 5 and its proof]{LacaSpi:purelyinf} by Laca and Spielberg. To deduce simplicity of $C_0(X) \rtimes_\red \Gamma$ one can also use Archbold--Spielberg,~\cite{ArcSpi:free}. 
Recall that a projection $p$ in a \Cs{} $\cA$ is said to be \emph{properly infinite} if $p \oplus p \precsim p$, i.e., $p \oplus p$ is Murray--von Neumann equivalent to a subprojection of $p$. 

\begin{proposition}[Archbold--Spielberg, Laca--Spielberg] \label{prop:pi}
Let $\Gamma$ be a countable group acting minimally and topologically freely on a metrizable totally disconnected locally compact Hausdorff space $X$. Suppose further that each non-zero projection in $C_0(X)$ is properly infinite in $C_0(X) \rtimes_\red \Gamma$. Then $C_0(X) \rtimes_\red \Gamma$ is simple and purely infinite. 
\end{proposition}

\noindent J.-L.\ Tu proved in~\cite{Tu:B-C-1} (see also~\cite{Tu:B-C-2}) 
that the \Cs{} associated with any amenable group\-oid belongs to the UCT class. The  crossed product $C_0(X) \rtimes_\red \Gamma$ can be identified with the \Cs{} of the groupoid $X \rtimes \Gamma$, which is amenable if  $\Gamma$ acts amenably on $X$. Hence $C_0(X) \rtimes_\red \Gamma$ belongs to the UCT class whenever $\Gamma$ acts amenably on $X$. Combining this deep theorem with Propositions~\ref{prop:simple} and~\ref{prop:pi} we get:

\begin{corollary} \label{cor:pi+simple}
Let $X$ be a metrizable totally disconnected locally compact Hausdorff space and let $\Gamma$ be a group acting on $X$. Then $C_0(X) \rtimes_\red \Gamma$ is a Kirchberg algebra in the UCT class if and only if the action of $\Gamma$ on $X$ is minimal, topologically free, amenable, and each non-zero projection in $C_0(X)$ is properly infinite in $C_0(X) \rtimes_\red \Gamma$.
\end{corollary}

\noindent It would be very desirable to replace the last condition on the projections in $C_0(X)$ with  
purely dynamical conditions on  the $\Gamma$-space $X$. This leads us to the notion of \emph{purely infinite actions}, which we shall proceed to define.

Let $\Gamma$ be a group acting on a locally compact totally disconnected Hausdorff space $X$. Let $\K(X)$ (or just $\K$ if $X$ is understood) be the algebra of all compact-open subsets of $X$.  A set $K \in \K$ is said to be  \emph{$(X,\Gamma, \K)$-paradoxical} if there exist pairwise disjoint sets $K_1, K_2, \dots, K_{n+m} \in  \K$ and elements $t_1, t_2, \dots,  t_{n+m} \in \Gamma$ such that $K_j \subseteq K$ for all $j$ and
\begin{equation} \label{eq:E}
K = \bigcup_{j=1}^n t_j . K_j = \bigcup_{j=n+1}^{n+m} t_j . K_j.
\end{equation}
(We can think of this version of paradoxicality as a relative version of the classical notion of paradoxicality from the Hausdorff--Banach--Tarski paradox.)

\begin{definition} \label{def:pi-action}
An action of a group $\Gamma$ on a totally disconnected Hausdorff space $X$ is said to be \emph{purely infinite} if every compact-open subset of $X$ is $(X,\Gamma,\K)$-paradoxical (in the sense defined above).
\end{definition}

\noindent Let us for a moment turn to the \Cs{} point of view. Let us first fix some (standard) notations for  crossed product \Cs s. Let $\Gamma$ be a group acting on a \Cs{} $\cA$ (where $\cA$ for example could be $C_0(X)$ for some locally compact Hausdorff space $X$). Let the action of $\Gamma$ on $\cA$ be denoted by $t \mapsto \alpha_t \in \mathrm{Aut}(\cA)$, $t \in \Gamma$. If $\cA = C_0(X)$ and $\Gamma$ acts on $X$, then the action of $\Gamma$ on $\cA$ is given by $\alpha_t(f) = t.f$, where $(t.f)(x) = f(t^{-1}.x)$ for $f \in C_0(X)$ and $x \in X$. Let $u_t$, $t \in \Gamma$, denote the unitary elements in (the multiplier algebra of) $\cA \rtimes_\red \Gamma$ that implement the action of $\Gamma$ on $\cA$, i.e., $\alpha_t(a) = u_t a u_t^*$ for $a \in \cA$ and $t \in \Gamma$. The set of finite sums of the form $\sum_{t \in \Gamma} a_t u_t$, where $a_t \in \cA$ and only finitely many $a_t$ are non-zero, forms a uniformly dense $^*$-subalgebra of $\cA \rtimes_\red \Gamma$.

A projection $p \in \cA$ is said to be \emph{$(\cA,\Gamma)$-paradoxical} if there are pairwise orthogonal subprojections $p_1, p_2, \dots, p_{n+m}$ of $p$ in $\cA$ and elements $t_1, t_2, \dots, t_{n+m}$ in $\Gamma$ such that
\begin{equation} \label{eq:A-paradoxical}
p = \sum_{k=1}^n \alpha_{t_k}(p_k) = \sum_{k=n+1}^{n+m} \alpha_{t_k}(p_k).
\end{equation}

We record some facts about paradoxical projections and sets:

\begin{itemize}
\item[(a)]  If $p$ is a projection in a \Cs{} $\cA$ on which a group $\Gamma$ acts, and if $p$ is $(\cA,\Gamma)$-paradoxical, then $p$ is properly infinite in $\cA \rtimes_\red \Gamma$.
\item[(b)] If $K$ is a compact-open subset of a locally compact Hausdorff $\Gamma$-space $X$, then $K$ is $(X,\Gamma,\K)$-paradoxical if and only if the projection $1_K$ is $(C_0(X),\Gamma)$-paradoxical. 
\item[(c)] If $A \subseteq \Gamma$, then $A$ is paradoxical if and only if $K_A$ is $(\beta \Gamma, \Gamma, \K)$-paradoxical, which again happens if and only if $1_A$ is a properly infinite projection in $\ell^\infty(\Gamma) \rtimes_\red \Gamma$ or, equivalently, if and only if $1_{K_A}$ is a properly infinite projection in $C(\beta \Gamma) \rtimes_\red \Gamma$.
\end{itemize}

\begin{proof} (a). Let $p_1, p_2, \dots, p_{n+m} \in \cA$ be pairwise orthogonal subprojections of $p$  and  $t_1, t_2, \dots, t_{n+m}$ in $\Gamma$ be such that \eqref{eq:A-paradoxical} holds. Let $u_t$ be the unitaries in the multiplier algebra of $\cA$ that implement the action $\alpha_t$, $t \in \Gamma$, on $\cA$. Define the following two elements in $\cA \rtimes_\red \Gamma$:
$$v = \sum_{j=1}^n p_j u_{t_j}^*, \qquad w = \sum_{j=n+1}^{n+m} p_j u_{t_j}^*.$$
It then follows from \eqref{eq:A-paradoxical} that
$$v^*v = p = w^*w, \qquad vv^* \perp ww^*, \qquad vv^* \le p \qquad ww^* \le p.$$ 
This shows that $p$ is a properly infinite projection in $C_0(X) \rtimes_\red \Gamma$.

(b). This follows immediately from the fact that there is a bijective correspondence between compact-open subsets of $X$ and projections in $C_0(X)$ (given by $K \leftrightarrow 1_K$). 

(c). The first claim, that $A$ is paradoxical if and only if $K_A$ is $(\beta \Gamma, \Gamma, \K)$-paradoxical, follows easily from Lemma~\ref{lem:KA}. Indeed, $A$ is paradoxical if and only if there exist pairwise disjoint subsets $A_1$, $A_2$, $\dots$, $A_{n+m} \subseteq A$ and elements $t_1,t_2, \dots, t_{n+m}$ such that 
\begin{equation} \label{eq:A}
A = \bigcup_{j=1}^n t_j  A_j = \bigcup_{j=n+1}^{n+m} t_j  A_j.
\end{equation}
Taking closures (relatively to $\beta \Gamma$) we obtain that \eqref{eq:E} holds with $K=K_A$ and $K_j = K_{A_j}$. Hence $K_A$ is $(\beta \Gamma, \Gamma, \K)$-paradoxical. For the reverse implication, if \eqref{eq:E} holds with $K=K_A$, then intersect all sets in \eqref{eq:E} with $\Gamma$ and use Lemma~\ref{lem:KA} to conclude that \eqref{eq:A} holds, so $A$ is paradoxical.

It was shown in~\cite[Proposition 5.5]{RorSie:ETDS} that if $A \subseteq \Gamma$, then $1_A$ is properly infinite in $\ell^\infty(\Gamma) \rtimes_\red \Gamma$ if and only if $A$ is paradoxical. The natural isomorphism from $\ell^\infty(\Gamma)$ to $C(\beta \Gamma)$ maps $1_A$ to $1_{K_A}$, and this isomorphism preserves the property of being properly infinite in the crossed product by $\Gamma$.
\end{proof}

\begin{lemma} \label{lm:3conditions}
Let $\Gamma$ be a group acting on a locally compact totally disconnected Hausdorff space $X$. 
Consider the following three conditions on a compact-open subset $K$ of $X$:
\begin{enumerate}
\item There is no $\Gamma$-invariant Radon measure $\mu$ on $X$ such that $\mu(K) >0$.
\item $1_K$ is a properly infinite projection in $C_0(X) \rtimes_\red \Gamma$.
\item $K$ is $(X,\Gamma,\K)$-paradoxical. 
\end{enumerate}
Then $\mathrm{(iii)} \Rightarrow \mathrm{(ii)} \Rightarrow \mathrm{(i)}$.
\end{lemma}

\begin{proof}
(ii) $\Rightarrow$ (i). Assume (i) does not hold. Then there is an invariant Radon measure $\mu$ on $X$ such that $\mu(K) = 1$. The measure $\mu$ extends to a densely defined (possibly unbounded) trace $\tau$ on $C_0(X) \rtimes_\red \Gamma$ (cf.~\cite[Lemma 5.3]{RorSie:ETDS}) such that $\tau(1_K) = \mu(K) = 1$. A properly infinite projection is either zero (or infinite) under any positive trace, so $1_K$ cannot be properly infinite in $C_0(X) \rtimes_\red \Gamma$.

(iii) $\Rightarrow$ (ii). This follows from (a) and (b) above.
\end{proof}

\noindent The proposition below follows immediately from Corollary~\ref{cor:pi+simple} and from ``(iii) $\Rightarrow$ (ii)'' of the lemma above. 

\begin{proposition} \label{prop:pi-actions}
Let $X$ be a metrizable totally disconnected locally compact Hausdorff space and let $\Gamma$ be a group which acts on $X$ in such a way that the action is topologically free, minimal, amenable and purely infinite. 
Then $C_0(X) \rtimes_\red \Gamma$ is a Kirchberg algebra in the UCT class.
\end{proposition}

\begin{remark}[The type semigroup] \label{rem:typesg}
It is an interesting question if the converse of Proposition~\ref{prop:pi-actions} holds, or to what extend the reverse implications in Lemma~\ref{lm:3conditions} hold. In other words, if $\Gamma$ acts freely, minimally and amenably on a metrizable totally disconnected locally compact Hausdorff space $X$ such that $C_0(X) \rtimes_\red \Gamma$ is a Kirchberg algebra, does it then follow that the action is purely infinite? As a Kirchberg algebra has no traces, the space $X$ cannot have any non-zero $\Gamma$-invariant Radon measures if $C_0(X) \rtimes_\red \Gamma$ is a Kirchberg algebra. 

We can analyse these questions using the relative type semigroup $S(X,\Gamma, \K)$ considered for example in~\cite[Section 5]{RorSie:ETDS}, see also~\cite{Wag:B-T}.  This semigroup can be defined as being the universal ordered abelian semigroup generated by elements $[K]$, with $K \in \K$, subject to the relations
$$[t.K] = [K], \quad [K \cup K'] = [K] + [K'] \; \text{if $K \cap K' = \emptyset$,} \quad [K] \le [K'] \; \text{if $K \subseteq K'$,}$$ 
where $K,K' \in \K$ and $t \in \Gamma$. As the set of compact-open sets is closed under forming set differences, we deduce that the ordering on $S(X,\Gamma, \K)$ is the \emph{algebraic order}: If $x,y \in S(X,\Gamma, \K)$, then $x \le y$ if and only if there exists $z \in  S(X,\Gamma, \K)$ such that $y = x+z$.

In the language of the type semigroup, $K \in \K$ is $(X,\Gamma,\K)$-paradoxical if and only if $2[K] \le [K]$. 

Property (i) in Lemma~\ref{lm:3conditions}  holds if and only if there exists a natural number $n$ such that $(n+1)[K] \le n[K]$, which in turns implies that $m[K] \le n[K]$ for all $m \ge 1$. We conclude that the implication ``(i) $\Rightarrow$ (iii)'' holds for all compact-open subsets $K$ of $X$ if $S(X,\Gamma, \E)$ is \emph{almost unperforated}, that is, for all $x,y \in  S(X,\Gamma,\K)$ for all integers $n \ge 1$, $(n+1)x \le ny$  implies $x \le y$. 
In fact,  ``(i) $\Rightarrow$ (iii)'' also holds under the much weaker comparability assumption that there exist integers $n,m \ge 2$ such that for all $x,y \in  S(X,\Gamma,\K)$, $nx \le my$  implies $x \le y$. 

It has recently been shown by Ara and Exel in~\cite{AraExel:paradoxical} that $S(X,\Gamma,\K)$ need not be almost unperforated. They give counterexamples where $\Gamma$ is a free group acting on the (locally compact, non-compact) Cantor set. It is not known if this phenomenon can occur also when the free group acts freely and minimally on the Cantor set. Nonetheless, the example by Ara and Exel indicates that the questions raised in this remark may all have negative answers. 
\end{remark}

\noindent We close this section with two lemmas that shall be used in Section~\ref{sec:Kirchberg}. An open subset $E$ of a $\Gamma$-space $X$ is said to be \emph{$\Gamma$-full} in $X$ if $X = \bigcup_{t \in \Gamma} t.E$.

\begin{lemma} \label{lm:extending}
Let $X$ be a totally disconnected locally compact Hausdorff space, let $\Gamma$ be a group acting co-compactly on $X$, and let $Y$ be a closed $\Gamma$-invariant subset of $X$. 
\begin{enumerate}
\item If $K'$ is a compact-open $\Gamma$-full subset of $Y$, then there exists a compact-open $\Gamma$-full subset $K$ of $X$ such that $K' = K \cap Y$.
\item If $K$ is a $(X,\Gamma,\K)$-paradoxical compact-open subset of $X$, then $K \cap Y$ is a $(Y,\Gamma,\K)$-paradoxical compact-open subset of $Y$.
\end{enumerate}
\end{lemma}

\begin{proof}
(i). There is a compact-open subset $K_0$ of $X$ such that $K' = K_0 \cap Y$. 
It follows from the assumption that the action of $\Gamma$ on $X$ is co-compact that there exists a compact-open subset $L \subseteq X$ such that $X = \bigcup_{t \in \Gamma} t.L$. Since $K'$ is $\Gamma$-full in $Y$ and $L \cap Y$ is compact there is a finite subset $S \subseteq \Gamma$ such that $L \cap Y \subseteq \bigcup_{t \in S} t.K'$. Put 
$$K = K_0 \cup \big(L \setminus \bigcup_{t \in S} t.K'\big).$$
Then $K$ is compact-open, $K \cap Y = K_0 \cap Y = K'$, and $X = \bigcup_{t \in \Gamma} t.K$. 

(ii). Follows from the definitions.
\end{proof}

\begin{lemma} \label{lm:clopen-paradoxical} Consider the canonical action of $\Gamma$ on $\beta \Gamma$. Let $A$ and $B$ be subsets of $\Gamma$. Then $K_B$ is a $\Gamma$-full subset of $X_A$ if and only if  $A \propto_\Gamma B \propto_\Gamma A$.
\end{lemma}

\begin{proof} By definition, $K_B$ is $\Gamma$-full in $X_A$ if and only if $X_B = X_A$. The lemma therefore follows from Lemma~\ref{lm:A-B-comp}(ii). 
\end{proof}

\section{The Roe algebra} \label{sec:Roe}

\noindent The Roe algebra associated with a group $\Gamma$ is the reduced crossed product \Cs{} 
$\ell^\infty(\Gamma) \rtimes_\red \Gamma$. 
We give below a characterisation of supramenability in terms of the Roe algebra.

\begin{lemma} \label{lm:equiv-propinf}
Let $p \in \ell^\infty(\Gamma)$ and $q \in \ell^\infty(\Gamma) \rtimes_\red \Gamma$ be projections that generate the same closed two-sided ideal in $\ell^\infty(\Gamma) \rtimes_\red \Gamma$. If $q$ is properly infinite, then so is $p$.
\end{lemma}

\begin{proof} By the assumption that $p$ and $q$ generate the same closed two-sided ideal in  the Roe algebra $\ell^\infty(\Gamma) \rtimes_\red \Gamma$ it follows that $p \precsim q \otimes 1_n$ and $q \precsim p \otimes 1_m$ relatively to  $\ell^\infty(\Gamma) \rtimes_\red \Gamma$ for some integers $n,m \ge 1$. As $q$ is properly infinite we have $q \otimes 1_k \precsim q$ for all $k \ge 1$. Hence
$$(p \otimes 1_m) \oplus (p \otimes 1_m) \precsim q \otimes 1_{2nm} \precsim q \precsim p \otimes 1_m,$$
so $p \otimes 1_m$ is properly infinite in the Roe algebra. It then follows from  
\cite[Proposition 5.5]{RorSie:ETDS} that $p$ itself is properly infinite in the Roe algebra.
\end{proof}

\begin{remark}[The Pedersen ideal] \label{rm:compact-ideals} 
Gert K.\ Pedersen proved that every \Cs{} $\cA$ contains a smallest dense two-sided ideal, now called the \emph{Pedersen ideal in $\cA$}. The Pedersen ideal of $\cA$ can be obtained as the intersection of all dense two-sided ideals in $\cA$; and the content of Pedersen's result is that this ideal is again a dense ideal in $\cA$. 

For every positive element $a$ in $\cA$ and for every $\ep > 0$ one can consider the ``$\ep$-cut-down'', $(a-\ep)_+ \in \cA$, (which is the positive part of the self-adjoint element $a - \ep \cdot 1_{\widetilde{\cA}}$ in the unitization $\widetilde{\cA}$ of $\cA$). This element $(a-\ep)_+$ belongs to the Pedersen ideal for every positive element $a$ in $\cA$ and for every $\ep > 0$. If $p \in \cA$ is a projection and if $0 < \ep < 1$, then $(p-\ep)_+ = (1-\ep)p$. All projections in $\cA$ therefore belong to the Pedersen ideal in $\cA$. 

The Pedersen ideal of a unital \Cs{} is the algebra itself, i.e., unital \Cs s have no proper dense two-sided ideals. The Pedersen ideal of $\cK(H)$, the compact operators on a Hilbert space $H$,  is the algebra of all finite rank operators on $H$.  The Pedersen ideal of $C_0(X)$, for some locally compact Hausdorff space $X$, is $C_c(X)$. 

An element $a$ in a \Cs{} $\cA$ is said to be \emph{full}, if $a$ is not contained in any proper closed two-sided ideal in $\cA$. If $\cA$ contains a full projection $p$, then the primitive ideal space of $\cA$ is (quasi-)compact. This can be rephrased as follows: Whenever $\{\cI_\alpha\}_{\alpha \in \mathbb{A}}$ is an increasing net of closed two-sided ideals in $\cA$ such that $\cA=\overline{\bigcup_{\alpha \in \mathbb{A}} \cI_\alpha}$, then $\cA= \cI_\alpha$ for some $\alpha \in \mathbb{A}$.  Indeed, $\bigcup_{\alpha \in \mathbb{A}} \cI_\alpha$ is a dense two-sided ideal in $\cA$, which therefore contains the Pedersen ideal, and hence contains all projections in $\cA$. Thus $p \in \cI_\alpha$ for some $\alpha$, whence $\cA = \cI_\alpha$. 
\end{remark}

\noindent As mentioned in the previous section, it was shown in~\cite[Proposition 5.5]{RorSie:ETDS} that if $A \subseteq \Gamma$, then $1_A$ is properly infinite in $\ell^\infty(\Gamma) \rtimes_\red \Gamma$ if and only if $A$ is paradoxical. We sharpen this result as follows:

\begin{proposition} \label{prop:Roe-Supramenable}
The following two conditions are equivalent for every group $\Gamma$:
\begin{enumerate}
\item $\Gamma$ is supramenable.
\item The Roe algebra $\ell^\infty(\Gamma) \rtimes_\red \Gamma$ contains no properly infinite projection.
\end{enumerate}
\end{proposition}

\begin{proof} Suppose that (ii) does not hold. If $\Gamma$ is non-exact, then $\Gamma$ is non-amenable and in particular not supramenable, so (i) does not hold. Suppose that $\Gamma$ is exact and that  $\ell^\infty(\Gamma) \rtimes_\red \Gamma$ contains a properly infinite projection $p$. Let $\cI$ be the closed two-sided ideal in $\ell^\infty(\Gamma) \rtimes_\red \Gamma$ generated by $p$. Then, by~\cite[Theorem 1.16]{Sie:MJM}, which applies because $\Gamma$ is assumed to be exact and the action of $\Gamma$ on $\ell^\infty(\Gamma)$ is free, $\cI$ is the closed two-sided ideal generated by $\ell^\infty(\Gamma) \cap \cI$. Let $\mathbb{A}$ be the directed net of all finite subsets of the set of all projections in $\ell^\infty(\Gamma) \cap \cI$; and for each $\alpha \in \mathbb{A}$ let $\cI_\alpha$ be the closed two-sided ideal in $\ell^\infty(\Gamma) \rtimes_\red \Gamma$ generated by $\alpha$. Then 
$$\cI = \overline{\bigcup_{\alpha \in \mathbb{A}} \cI_\alpha}.$$
It therefore follows from Remark~\ref{rm:compact-ideals} that  $\cI= \cI_\alpha$ for some $\alpha \in \mathbb{A}$. Let $q \in \ell^\infty(\Gamma)$ be the supremum of the projections belonging to $\alpha$. Then $\cI$ is equal to the closed two-sided ideal generated by $q$. Hence $p$ and $q$ generate the same closed two-sided ideal in the Roe algebra, whence $q$ is properly infinite by Lemma~\ref{lm:equiv-propinf}.

If (ii) holds and if $A \subseteq \Gamma$, then $1_A$ is not a properly infinite projection in $\ell^\infty(\Gamma) \rtimes_\red \Gamma$, so $A$ is non-paradoxical, cf.\ ~\cite[Proposition 5.5]{RorSie:ETDS} or claim (c) (below Definition~\ref{def:pi-action}). This shows that $\Gamma$ is supramenable. 
\end{proof}

\begin{remark} A projection is said to be infinite if it is Murray--von Neumann equivalent to a proper subprojection of itself. A unital \Cs{} is infinite if its unit is an infinite projection, or, equivalently, if it contains a non-unitary isometry. If a unital \Cs{} contains an infinite projection, then it is infinite itself.  The Roe algebra $\ell^\infty(\Gamma) \rtimes_\red \Gamma$ is therefore infinite whenever $\Gamma$ is non-supramenable. 

The Roe algebra $\ell^\infty(\Gamma) \rtimes_\red \Gamma$ is also infinite whenever $\Gamma$ contains an element $t$ of infinite order. Indeed, let $A = \{t^n \mid n \ge 0\}$. Then 
$tA \subset A$, so
$u_t 1_A u_t^* = 1_{tA} < 1_A$, which shows that $1_A$ is an infinite projection in  $\ell^\infty(\Gamma) \rtimes_\red \Gamma$. 
(As before we let $(u_t)_{t \in \Gamma}$ denote the unitaries in $\ell^\infty(\Gamma) \rtimes_\red \Gamma$ implementing the action of $\Gamma$ on $\ell^\infty(\Gamma)$.)

On the other hand, if $\Gamma$ is locally finite (an increasing union of finite groups), then  $\ell^\infty(\Gamma) \rtimes_\red \Gamma$ is finite. To see this observe first that $C(X) \rtimes_\red \Gamma$ is finite whenever $\Gamma$ is a finite group acting on a compact Hausdorff space $X$. Indeed, a unital \Cs{} is finite if it admits a separating family of tracial states. A crossed product \Cs{} $C(X) \rtimes_\red \Gamma$ is therefore finite if $X$ admits a separating family of invariant probability measures. (By separating we mean that every non-empty open set is non-zero on at least one of the probability measures in the family.) If $\Gamma$ is finite then the family consisting of $\Gamma$-means of every probability measure on $X$ will be such a separating family of invariant probability measures.

Suppose $\Gamma$ is locally finite and write $\Gamma = \bigcup_{n=1}^\infty \Gamma_n$, where $\Gamma_1 \subseteq \Gamma_2 \subseteq \Gamma_3 \subseteq \cdots$ is an increasing sequence of finite subgroups of $\Gamma$. Then 
$$ \ell^\infty(\Gamma) \rtimes_\red \Gamma =  
\varinjlim \,  \ell^\infty(\Gamma) \rtimes_\red \Gamma_n,$$
since $\ell^\infty(\Gamma) \rtimes_\red \Gamma_n$ is isomorphic to $C^*\big(\ell^\infty(\Gamma) \cup \{u_t \mid t \in \Gamma_n\}\big)$. Any inductive limit of finite \Cs s is again finite, so $ \ell^\infty(\Gamma) \rtimes_\red \Gamma$ is finite because all  $\ell^\infty(\Gamma) \rtimes_\red \Gamma_n$  are finite.

It seems plausible that the Roe algebra $\ell^\infty(\Gamma) \rtimes_\red \Gamma$ is finite if and \emph{only if} $\Gamma$ is locally finite.
\end{remark}

\noindent
In the proposition below  we identify the Roe algebra with $C(\beta \Gamma) \rtimes_\red \Gamma$.

If $A \subseteq \Gamma$ and $\{U_i\}_{i \in I}$ is an increasing family of proper open $\Gamma$-invariant subsets of $X_A$, then $\bigcup_{i \in I} U_i$ is also a proper subset of $X_A$. (Otherwise $K_A$ would be contained in one of the $U_i$'s, which again would entail that $U_i = X_A$.) It follows that $X_A$ contains a maximal proper open $\Gamma$-invariant subset, or, equivalently, that $X_A$ contains a minimal (non-empty) closed $\Gamma$-invariant subset. This argument shows moreover that each non-emtpy closed invariant subset of $X_A$ contains a minimal (non-empty) closed $\Gamma$-invariant subset. Recall from Definition~\ref{def:small} the definition of the notion $A \ll^* \Gamma$.

\begin{proposition} \label{prop:X_A} Let $A$ be a non-empty subset of $\Gamma$ and let $Z$ be a minimal (non-empty) closed $\Gamma$-invariant subset of $X_A$. Then:  
\begin{enumerate}
\item $Z$ is a locally compact totally disconnected Hausdorff space, and the action of $\Gamma$ on $Z$ is free and minimal.
\item The action of $\Gamma$ on $Z$ is purely infinite if $A$ is paradoxical.
\item The action of $\Gamma$ on $Z$ is amenable if $\Gamma$ is exact.
\item If $A \ll^* \Gamma$, then $Z$ is necessarily non-compact; and if $A \ll^* \Gamma$ does not hold, then there exists a minimal (non-empty) closed $\Gamma$-invariant subset of $X_A$ which is compact.\footnote{Added in proof: The existence of one compact minimal closed invariant subset of $X_A$ does not in general imply that all minimal closed invariant subset of $X_A$ are compact.}
\end{enumerate}
\end{proposition}

\begin{proof} It is clear from its definition that $Z$ is a closed $\Gamma$-invariant subspace of $X_A$, and hence itself a $\Gamma$-space. Being a closed subset of the totally disconnected locally compact Hausdorff space $X_A$, $Z$ is also a totally disconnected locally compact Hausdorff space. By maximality of $U$, the action of $\Gamma$ on $Z$ is minimal. 

Freeness and amenability of an action of a group on a space pass to any $\Gamma$-invariant subspace.  As $\Gamma$ acts freely on $\beta \Gamma$ for all $\Gamma$ we conclude that (i) holds. Moreover, if $\Gamma$ is exact, then the action of $\Gamma$ on $\beta \Gamma$ is amenable  (see~\cite[Theorem 5.1.6]{BroOza:book}), so (iii) holds. 

Let us prove (ii). Assume that $A$ is paradoxical. Take a non-empty compact-open subset $K$ of $Z$. Then $K$ is $\Gamma$-full in $Z$ by minimality of the action. We can therefore use Lemma~\ref{lm:extending}(i) to find a compact-open subset $K'$ of $X_A$ which is $\Gamma$-full in $X_A$ and satisfies $K=K' \cap Z$. Put $B = K' \cap \Gamma$. Then $K_B = K'$ by Lemma~\ref{lem:KA}(iv), and $A \propto_\Gamma B  \propto_\Gamma A$ by Lemma~\ref{lm:clopen-paradoxical}. Hence $B$ is paradoxical by Lemma~\ref{lm:A-B-comp}(ii). Claim (c) (below Definition~\ref{def:pi-action}) then states that  $K'=K_B$ is $(X_A,\Gamma,\K)$-paradoxical, and Lemma~\ref{lm:extending}(ii) finally yields that $K$ is $(Z,\Gamma,\K)$-paradoxical.

(iv). Suppose that $Z$ is compact. As $Z \subseteq X_A = \bigcup_{t \in \Gamma} t.K_A$, it follows that there exists a finite set $S \subseteq \Gamma$ such that $Z \subseteq \bigcup_{s \in S} s.K_A = K_B$, where $B = \bigcup_{s \in S} sA$. For each finite set $T \subseteq \Gamma$ we have
$$Z = \bigcap_{t \in T} t.Z \subseteq \bigcap_{t \in T} t.K_B = K_{\bigcap_{t \in T} tB}.$$
This shows that $\bigcap_{t \in T} tB$ is non-empty, so $B$ is absorbing. Hence $A \ll^* \Gamma$ does not hold. 

Suppose, conversely, that $A \ll^* \Gamma$ does not hold. Find an absorbing subset $B$ of $\Gamma$ so that $B \propto A$. Then $K_B \subseteq X_A$ by Lemma~\ref{lm:A-B-comp}(i). Put
$W = \bigcap_{t \in \Gamma} t.K_B$.
Then $W$ is compact and $\Gamma$-invariant. Moreover, since
$$\bigcap_{s \in S} t.K_B = K_{\bigcap_{s \in S} sB} \ne \emptyset$$
for all finite sets $S \subseteq \Gamma$, we conclude that $W$ is non-empty. By the argument above the proposition we can find a minimal (non-empty) closed $\Gamma$-invariant subset $Z$ of $W$. As $W$ is compact, so is $Z$.
\end{proof}

\noindent It should be remarked that the space $Z$ from Proposition~\ref{prop:X_A} may not be metrizable and that the crossed product $C_0(Z) \rtimes_\red \Gamma$ accordingly may not  be separable. It follows from Proposition~\ref{prop:X_A} and from Proposition~\ref{prop:simple} that the crossed product $C_0(Z) \rtimes_\red \Gamma$ is simple and nuclear if $\Gamma$ is exact; and it follows from Proposition~\ref{prop:X_A} and Proposition~\ref{prop:pi-actions}  that the crossed product is simple, nuclear and purely infinite (but not necessarily a Kirchberg algebra) if $A$ moreover is paradoxical.

Recall from Lemma~\ref{lm:A-B-comp} that $X_A$ itself is compact if and only if $X_A = \beta \Gamma$. Clearly, any minimal closed $\Gamma$-invariant subset of $\beta \Gamma$ is compact. 
However, a minimal closed invariant subset $Z$ of $X_A$ can be compact even when $X_A$ is non-compact. This follows from  Proposition~\ref{prop:X_A}(iv) if we can find an absorbing subset $A$ of a group $\Gamma$ such that $X_A \ne \beta \Gamma$ (or equivalently, such that $\Gamma$ is not $A$-bounded). There are many such examples of $A \subseteq \Gamma$, eg.\  $\Gamma = \Z$ and $A = \N$.

Perhaps surprisingly it turns out to a subtle matter to decide when the space $Z$ is non-discrete. Clearly, if $A$ is non-empty and finite, then $Z = X_A = \Gamma$ and we are in the trivial situation of $\Gamma$ acting on itself. More generally, if $Z$ has an isolated point, then the $\Gamma$-space $Z$ is conjugate to $\Gamma$. (Indeed, if $x_0 \in Z$ is an isolated point, then  $Z = \Gamma.x_0$ by minimality of the action.)

The example below shows that $Z$ (in Proposition~\ref{prop:X_A}) can be discrete (and hence trivial) even when $A$ is infinite. More precisely, if $A \subseteq \Gamma$ is such that $sA \cap A$ is finite for all $e \ne s \in \Gamma$, then $Z$ is discrete. On the other hand, if $A$ is paradoxical, then $Z$ cannot be discrete by Proposition~\ref{prop:X_A}(ii).

\begin{example} \label{ex:discrete}
Let $\Gamma$ be a countably infinite group and suppose that $A \subseteq \Gamma$ is an infinite set such that $|sA \cap A| < \infty$ for all $s \in \Gamma \setminus \{e\}$. It was shown in Lemma~\ref{lm:smallA2} (and its proof) that each infinite group $\Gamma$ contains such a subset $A$. Let $1_A \in \ell^\infty(\Gamma) \subseteq \ell^\infty(\Gamma) \rtimes_\red \Gamma$ denote indicator function for $A$. 

Inside the Roe algebra, $c_0(\Gamma) \rtimes_\red \Gamma$ is the smallest non-zero ideal (every other non-zero ideal contains this ideal). If $A \subseteq \Gamma$ is as above, then the following holds in the quotient of the Roe algebra by  $c_0(\Gamma) \rtimes_\red \Gamma$:
\begin{equation} \label{eq:a}
\forall x \in \ell^\infty(\Gamma) \rtimes_\red \Gamma:  1_A \, x \, 1_A + c_0(\Gamma) \rtimes_\red \Gamma \; \in \; \ell^\infty(\Gamma) + c_0(\Gamma) \rtimes_\red \Gamma.
\end{equation}
In other words, if $\pi \colon  \ell^\infty(\Gamma) \rtimes_\red \Gamma \to \big(\ell^\infty(\Gamma)/c_0(\Gamma)\big) \rtimes_\red \Gamma$ denotes the quotient mapping, then $\pi(1_A)$ is an abelian projection. In particular, the corner algebra $1_A (\ell^\infty(\Gamma) \rtimes_\red \Gamma)1_A$ has a character. 

Let us see why \eqref{eq:a} holds. It suffices to establish \eqref{eq:a}  for $x$ in a dense subset of $\ell^\infty(\Gamma) \rtimes_\red \Gamma$, so we may assume that $x = \sum_{t \in F} f_t u_t$, where $F \subset \Gamma$ is finite, $f_t \in \ell^\infty(\Gamma)$, and $t \mapsto u_t$ is the canonical unitary representation of the action of $\Gamma$ on $\ell^\infty(\Gamma)$. Now,
$$1_Ax1_A = \sum_{t \in F} 1_Af_t  u_t1_A =  \sum_{t \in F} 1_A \, f_t \,1_{tA} u_t =  \sum_{t \in F} 1_{A \cap tA} \, f_t u_t.$$
But $1_{A \cap tA}  \in c_0(\Gamma)$ whenever $t \ne e$, so $\pi(1_Ax1_A) = \pi(1_A f_e)$ and $1_A f_e \in \ell^\infty(\Gamma)$.

Let $A\subseteq \Gamma$ be as above,  consider the open set $X_A$, and let $U$ be a maximal proper open $\Gamma$-invariant subset of $X_A$, cf.\ Definition~\ref{def:KA-XA}. Since $A$ is infinite and $\Gamma$ is dense in $\beta \Gamma$, we have $\Gamma \subseteq U \subset X_A$.  We claim that the minimal $\Gamma$-space $Z = X_A \setminus U$ is, in fact, discrete. More precisely, the compact-open subset $K_A \cap Z$  of $Z$ is a singleton, and hence an isolated point in $Z$.  
Indeed,  the natural epimorphism $C_0(X_A) \rtimes_\red \Gamma \to C_0(Z) \rtimes_\red \Gamma$ maps $1_{K_A}$ onto $1_{K_A \cap Z}$ and its kernel contains $c_0(\Gamma) \rtimes_\red \Gamma$. We can therefore deduce from \eqref{eq:a} that
$$\cA := 1_{K_A \cap Z} \big(  C_0(Z) \rtimes_\red \Gamma \big) 1_{K_A \cap Z}$$
is abelian. As $C_0(Z) \rtimes_\red \Gamma$ is a simple \Cs, so is $\cA$, whence $\cA  \cong \C$. As $C(K_A\cap Z)$ is a sub-\Cs{} of $\cA$ this implies that $K_A \cap Z$ is a singleton.
\end{example}

\section{Kirchberg algebras arising as crossed products by non-supramenable groups} \label{sec:Kirchberg}

\noindent We show here that every  non-supramenable countable group admits a free, minimal, purely infinite\footnote{In the sense of Definition~\ref{def:pi-action}} action on the locally compact, non-compact Cantor set; and that the action moreover can be chosen to be amenable if the group is exact.  The (reduced) crossed product \Cs{} associated with such an action will in the latter case be a stable Kirchberg algebra in the UCT class. Our construction is an adaption of the one from~\cite[Section 6]{RorSie:ETDS}, where it was shown that every exact non-amenable countable group admits a free, minimal, amenable action on the (compact) Cantor set such that the crossed product \Cs{} is a unital Kirchberg algebra in the UCT class. (One can easily modify the construction in~\cite{RorSie:ETDS} to make the action on the Cantor set purely infinite.) 

Fix a countable group $\Gamma$ and a non-empty subset $A$ of  $\Gamma$. Let $K_A \subseteq X_A \subseteq \beta \Gamma$ be as in Definition~\ref{def:KA-XA}.  Note that $C_0(X_A)$ is a $\Gamma$-invariant closed ideal in $C(\beta \Gamma)$, and the smallest such which contains the projection $1_{K_A}$. 
Sometimes we prefer to work with $\ell^\infty(\Gamma)$ rather than $C(\beta \Gamma)$. To avoid confusion we denote by $\cC_A$ the $\Gamma$-invariant ideal in $\ell^\infty(\Gamma)$ that corresponds to $C_0(X_A)$. Thus $\cC_A$  is the smallest $\Gamma$-invariant closed ideal in $\ell^\infty(\Gamma)$ which contains $1_A$.  

Fix an increasing sequence $\{F_n\}_{n=1}^\infty$ of finite subsets of $\Gamma$ with $\bigcup_{n=1}^\infty F_n = \Gamma$. Put $B_n = \bigcup_{t \in F_n} tA$, and put $p_n = 1_{B_n} \in \cC_A$. Then $\{p_n\}_{n=1}^\infty$ is an increasing approximate unit for $\cC_A$ consisting of projections. We shall use and refer to this approximate unit several times in the following.  

It is well-known that the canonical action of $\Gamma$ on $\beta \Gamma$ is free, and it was shown by Ozawa (see~\cite[Theorem 5.1.6]{BroOza:book}) that this action is amenable whenever $\Gamma$ is exact. The two lemmas below tell us how these properties can be preserved after passing to a suitable separable sub-\Cs{} of $\cC_A$. If $\cA$ is an abelian \Cs, then let $\widehat{\cA}$ denote its space of characters, so that   $\cA \cong C_0(\widehat{\cA})$. 

It was remarked in ~\cite{RorSie:ETDS} that if a group $\Gamma$ acts on a compact Hausdorff space $X$, and if there are projections $q_{j}^{(t)}$ in $C(X)$, for $t \in \Gamma \setminus \{e\}$ and for $j=1,2,3$, such that 
\begin{equation} \label{eq:q_t}
\alpha_t(q_j^{(t)}) \perp q_j^{(t)}, \qquad q_1^{(t)} + q_2^{(t)} + q_3^{(t)} = 1,
\end{equation}
for all $e \ne t \in \Gamma$ and $j=1,2,3$, then $\Gamma$ acts freely on $X$. 
(As below Definition~\ref{def:pi-action}, $t \mapsto \alpha_t$ denotes the induced action on $C(X)$ given by $\alpha_t(f)(x) = f(t^{-1}.x)$.) Consider now the case where $\Gamma$ acts on a \emph{non-compact} locally compact Hausdorff space $X$ and where $C_0(X)$ has an increasing approximate unit $\{p_n\}_{n=1}^\infty$ consisting of projections. It is then easy to see that $\Gamma$ acts freely on $X$ if there are projections $q_{j,n}^{(t)}$ in $C_0(X)$, for $t \in \Gamma \setminus \{e\}$, $j=1,2,3$ and $n \in \N$, such that 
\begin{equation} \label{eq:q_tn}
\alpha_t(q_{j,n}^{(t)}) \perp q_{j,n}^{(t)}, \qquad q_{1,n}^{(t)} + q_{2,n}^{(t)} + q_{3,n}^{(t)} = p_n,
\end{equation}
holds for all $e \ne t \in \Gamma$, $j=1,2,3$ and $n \in \N$.

\begin{lemma}[cf.\ {\cite[Lemma 6.3]{RorSie:ETDS}}]
\label{lm:free}
There is a countable subset $M'$ of $\cC_A$ such that if $\cA$ is any $\Gamma$-invariant sub-\Cs{} of $\cC_A$ which contains $M'$, then $\{p_n\}_{n=1}^\infty$ (defined above) is an approximate unit for $\cA$, and $\Gamma$ acts freely on the character space $\widehat{\cA}$ of $A$.
\end{lemma}

\begin{proof}  By~\cite[Corollary 6.2]{RorSie:ETDS} we can find projections 
$q_{j}^{(t)}$ in $\ell^\infty(\Gamma)$ for $e \ne t \in \Gamma$ and  $j=1,2,3$, such that \eqref{eq:q_t} holds. The projections
$$q_{j,n}^{(t)}= q_{j}^{(t)}p_n, \qquad e \ne t \in \Gamma, \quad j=1,2,3, \quad n \in \N,$$ 
satisfy \eqref{eq:q_tn}. 
As $p_n$ belongs to $\cC_A$ for each $n$, so does each $q_{j,n}^{(t)}$. Let 
$$M'= \big\{q_{j,n}^{(t)} \mid t \in \Gamma, \; n \in \N,  \; j=1,2,3\big\} \subseteq \cC_A,$$
and let $\cA$ be any $\Gamma$-invariant sub-\Cs{} of $\cC_A$ that contains $M'$.  Then each $p_n$ belongs to $\cA$ by \eqref{eq:q_tn}. Moreover,  $\{p_n\}_{n=1}^\infty$ is an approximate unit for $\cA$ because it is an approximate unit for $\cC_A$. Finally, $\Gamma$ acts freely on $\widehat{\cA}$ because \eqref{eq:q_tn} holds.
\end{proof}

\begin{lemma}
\label{lm:amenable} Suppose that $\Gamma$ is an exact countable group. Then
there is a countable subset $M''$ of $\cC_A$ such that if $\cA$ is any $\Gamma$-invariant sub-\Cs{} of $\cC_A$ which contains $M''$, then $\Gamma$ acts amenably on $\widehat{\cA}$.
\end{lemma}

\begin{proof} The action of $\Gamma$ on $\beta \Gamma$ is amenable, cf.~\cite[Theorem 5.1.6]{BroOza:book}. Let $(m_i)_{i \in I}$ be a net of approximately invariant continuous means
$m_i \colon \beta \Gamma \to \mathrm{Prob}(\Gamma)$ that witnesses the amenability of this action, cf.\ Definition~2.1 in~\cite{Ana:TAMS2002}. Let $\widetilde{m}_i$ be the restriction of $m_i$ to $X_A \subseteq \beta \Gamma$. Then $(\widetilde{m}_i)_{i \in I}$ is a net of  approximate invariant continuous means $\widetilde{m}_i \colon X_A \to \mathrm{Prob}(\Gamma)$ that witnesses the amenability of the action of $\Gamma$ on $X_A$. As $\Gamma$ is countable and $X_A$ is $\sigma$-compact, the approximate mean $(\widetilde{m}_i)_{i \in I}$ can be chosen to be countable.

The spectrum $Y:= \widehat{\cA}$ of a $\Gamma$-invariant sub-\Cs{} $\cA$ of $\cC_A$ is a $\Gamma$-invariant quotient of the space $X_A = \widehat{\cC_A}$. We wish to make sure that each of the functions $\widetilde{m}_i$ passes to this quotient to yield a continuous function $\widehat{m}_i \colon Y \to \mathrm{Prob}(\Gamma)$, in which case the net $(\widehat{m}_i)_{ \in I}$  will witness that $\Gamma$ acts amenably on $Y$. This will happen if for all $x,y \in X_A$, for which $\widetilde{m}_i^x \ne \widetilde{m}_i^y$ for some $i$, there exists $f \in \cA$ such that $f(x) \ne f(y)$.

For each $i$ and for each $t \in \Gamma$ let  $f_{i,t} \colon X_A \to \R$ be the  continuous function given by $f_{i,t}(x) = \widetilde{m}_i^x(\{t\})$, $x \in X_A$. For each $s \in \Gamma$ let $f_{i,t,s}$ be the restriction of $f_{i,t}$ to the compact-open subset $s.K_A$ of $X_A$.  The countable set 
$$M'' = \big\{f_{i,t,s} \mid i \in I, \; s,t \in \Gamma\big\},$$ will then separate any pair of points $x,y \in X_A$ that are separated by the $\widetilde{m}_i$'s.
\end{proof}

\noindent Recall the definition of paradoxicality of projections from Section~\ref{sec:C*}. 

\begin{lemma}[cf.\ {\cite[Lemma 6.6]{RorSie:ETDS}}]  \label{lm:prop-inf}
For each projection $p \in \cC_A$, which is $(\ell^\infty(\Gamma),\Gamma)$-paradoxical, there is a finite set $M_p \subseteq \cC_A$ such that whenever $\cA$ is a $\Gamma$-invariant sub-\Cs{} of $\cC_A$ which contains $M_p \cup \{p\}$, then $p$ is $(\cA,\Gamma)$-paradoxical. 
\end{lemma}

\begin{proof} Note first that if $p$ in  $\cC_A$ is $(\ell^\infty(\Gamma),\Gamma)$-paradoxical, then it is also $(\cC_A,\Gamma)$-paradoxical. Indeed, the projections $p_j \in \ell^\infty(\Gamma)$ that witness that $p$ is $(\ell^\infty(\Gamma),\Gamma)$-paradoxical, cf.\ \eqref{eq:A-paradoxical}, will automatically belong to the closed ideal $\cC_A$ in $\ell^\infty(\Gamma)$. Accordingly we can let $M_p$ consist of the finitely many projections $p_j$ from \eqref{eq:A-paradoxical}.
\end{proof}

\begin{lemma} \label{lm:paradox-proj} Each $\Gamma$-full projection in $\cC_A$ is $(\ell^\infty(\Gamma),\Gamma)$-paradoxical if $A \subseteq \Gamma$ is paradoxical.
\end{lemma}

\begin{proof} Take a $\Gamma$-full projection in $\cC_A \subseteq \ell^\infty(\Gamma)$ and write it as $1_B$ where $B \subseteq \Gamma$. Identify $1_B$ with the $\Gamma$-full projection $1_{K_B}$ in $C_0(X_A)$. Then $K_B$ is $\Gamma$-full in $X_A$, and so $A \propto_\Gamma B \propto_\Gamma A$ by Lemma~\ref{lm:clopen-paradoxical}. Hence $B$ is paradoxical by Lemma~\ref{lm:A-B-comp}(iii), so  $1_B$ is a paradoxical projection, cf.\ claim (c) below Definition~\ref{def:pi-action}.
\end{proof}

\begin{lemma}[cf.\ {\cite[Lemma 6.7]{RorSie:ETDS}}] \label{lm:projection}
Let $T$ be a countable subset of $\cC_A$. Then there is a countable $\Gamma$-invariant set $Q$ consisting of projections in $\cC_A$ such that $T \subseteq C^*(Q)$.
\end{lemma}

\begin{proof} The proof is identical with the proof of~\cite[Lemma 6.7]{RorSie:ETDS}. Use that $\cC_A$ is of real rank zero (being a closed two-sided ideal in the real rank zero \Cs{} $\ell^\infty(\Gamma)$).
\end{proof}

\noindent If the group $\Gamma$ acts on a \Cs{} $\cA$ and if $p \in \cA$ is a projection, then we say that $p$ is \emph{$\Gamma$-full} if $p$ is not contained in any proper closed two-sided $\Gamma$-invariant ideal in $\cA$. 

In the rest of this section we let $t \mapsto \alpha_t$, $t \in \Gamma$, denote the canonical actions of $\Gamma$ on $\ell^\infty(\Gamma)$, respectively, on $C(\beta \Gamma)$, as well as on all their $\Gamma$-invariant sub-\Cs s.

\begin{proposition} \label{prop:c}
Let $\Gamma$ be a non-supramenable group, and let $A$ be a subset of $\Gamma$. As above, let $\mathcal{C}_A$ be the smallest closed $\Gamma$-invariant ideal in $\ell^\infty(\Gamma)$ that contains $1_A$. It follows that there is a separable, $\Gamma$-invariant sub-\Cs{} $\mathcal{A}$ of $\mathcal{C}_A$ such that
\begin{enumerate}
\item $1_A \in \ell^\infty(\Gamma)$ is a $\Gamma$-full projection  in $\cA$.
\item $\mathcal{A}$ is generated as a \Cs{} by its projections. 
\item Every projection in $\mathcal{A}$, which is $(\ell^\infty(\Gamma),\Gamma)$-paradoxical, is also $(\mathcal{A}, \Gamma)$-paradoxical,
\item $\Gamma$ acts freely on the character space $\widehat{\cA}$ of $A$.
\end{enumerate}
Furthermore, if $\Gamma$ is exact, then $\mathcal{A}$ can be chosen such that $\Gamma$ acts amenably on $\widehat{\cA}$.
\end{proposition}

\begin{proof} If $\Gamma$ is exact, then let $M = M' \cup M'' \cup \{1_A\}$, and let otherwise $M = M' \cup \{1_A\}$, where $M'$ and $M''$ are as in  Lemma~\ref{lm:free} and Lemma~\ref{lm:amenable}, respectively. Suppose that $\cA$ is a $\Gamma$-invariant (separable) sub-\Cs{} of $\cC_A$ which contains $M$. Then $\Gamma$ acts freely on $\widehat{\cA}$, the approximate unit  $\{p_n\}_{n=1}^\infty$ defined above Lemma~\ref{lm:free}  is contained in and is an approximate unit for $\cA$ by  Lemma~\ref{lm:free}; and if $\Gamma$ is exact, then $\Gamma$ acts amenably on $\widehat{\cA}$ by Lemma~\ref{lm:amenable}.

Let us show that $1_A$ is a $\Gamma$-full projection in any such \Cs{} $\cA$. Let $\cI$ be the smallest closed $\Gamma$-invariant ideal in $\cA$ which contains $1_A$. Then $1_{tA} = \alpha_t(1_A) \in \cI$ for all $t \in \Gamma$. Since $p_n = 1_{B_n}$ and $B_n = \bigcup_{t \in F_n} tA$ for some finite subset $F_n$ of $\Gamma$, we see that $p_n \le \sum_{t \in F_n} 1_{tA} \in \cI$, so $p_n \in \cI$ for all $n$. As $\{p_n\}_{n=1}^\infty$ is an approximate unit for $\cA$ we conclude that $\cI = \cA$, so $1_A$ is $\Gamma$-full in $\cA$.

Construct inductively countable $\Gamma$-invariant sets of projections
$$Q_0 \subseteq Q_1 \subseteq Q_2 \subseteq \cdots \subseteq \cC_A, \qquad P_0 \subseteq P_1 \subseteq P_2 \subseteq \cdots \subseteq \cC_A,$$
such that $M \subseteq C^*(Q_0)$, $P_n$ is the set of projections in $C^*(Q_n)$, and such that each projection in $P_n$, which is $(\ell^\infty(\Gamma), \Gamma)$-paradoxical, is $(C^*(Q_{n+1}), \Gamma)$-paradoxical.  The set of projections in any separable abelian \Cs{} is countable. Therefore $P_n$ is countable if $Q_n$ is countable. 

The existence of $Q_0$ follows from Lemma~\ref{lm:projection}. Suppose that $n \ge 0$ and $Q_n, P_n \subseteq \cC_A$ have been found. The existence of $Q_{n+1}$ then follows from  Lemma~\ref{lm:prop-inf} and Lemma~\ref{lm:projection}. 

Put
$$P_\infty = \bigcup_{n=1}^\infty P_n, \qquad \cA = C^*(P_\infty).$$
Then $\cA$ is a $\Gamma$-invariant separable sub-\Cs{} of $\cC_A$ which contains $M$. By definition, $\cA$ is generated by its projections, so (ii) holds.

Observe that each projection in $\cA$ actually belongs to $P_\infty$. Indeed,  if $p$ is a projection in $\cA$, then, for some $n$, there is a projection $p' \in C^*(P_n)$ such that $\|p-p'\| < 1$. Since $\cA$ is commutative this entails that $p = p'$. Each projection in $C^*(P_n)$ belongs to $P_{n+1}$, so $p \in P_{n+1}$.

Let $p$ be a projection in $\cA$ which is $(\ell^\infty(\Gamma),\Gamma)$-paradoxical. Then $p \in P_n$ for some $n$. It follows that $p$ is $(C^*(Q_{n+1}),\Gamma)$-paradoxical and hence also $(\cA,\Gamma)$-paradoxical. Thus (iii) holds.
\end{proof}

\noindent Recall the definition of the relation $A \ll \Gamma$ from Definition~\ref{def:small}.

\begin{lemma} \label{lm:proper-paradoxical}
Let $\Gamma$ be a countable group and let $A$ be a non-empty subset of $\Gamma$ such that $A \ll \Gamma$. Let $\cA \subseteq \cC_A$ be as in Proposition~\ref{prop:c}. Then $\cA$ is non-unital, and so is $\cA/\cI$ for every proper closed  $\Gamma$-invariant ideal $\cI$ in $\cA$.
\end{lemma}

\begin{proof}  As in the beginning of this section (and as in the proof of Proposition~\ref{prop:c} above) let $\{p_n\}$ be the approximate unit for $\cC_A$ and for $\cA$, given by $p_n = 1_{B_n}$, where $B_n = \bigcup_{t \in F_n} tA$ for some increasing sequence, $\{F_n\}$, of finite subsets of $\Gamma$ whose union is $\Gamma$.

Let $\cI$ be a $\Gamma$-invariant closed two-sided ideal in $\cA$ and suppose that $\cA/\cI$ is unital. Then $p_n +\cI$ is a unit for $\cA/\cI$ for some $n$. Now, $A \precsim_\Gamma \Gamma \setminus B_n$ by the assumption that $A \ll \Gamma$. Hence there is a partition 
$A_1, \dots, A_k$ of $A$ and elements $t_1, \dots, t_k$ in $\Gamma$ such that $\{t_jA_j\}$ are pairwise disjoint subsets of $\Gamma \setminus B_n$.
For each $j$, 
$$0 = (1_{t_jA_j}+\cI)(1_{B_n}+\cI) = (1_{t_jA_j}+\cI)(p_n+\cI) =1_{t_jA_j} + \cI.$$ Hence $1_{t_jA_j}$ belongs to $\cI$ for all $j$,  so $1_A = \sum_{j=1}^n \alpha_{t_j^{-1}} (1_{t_jA_j})$ belongs to $\cI$. Finally, because $1_A$ is a $\Gamma$-full projection in $\cA$, we must have $\cI=\cA$.
\end{proof}

\noindent
By the \emph{locally compact non-compact Cantor set} $\Can^*$ we shall mean the unique (up to homeomorphism) locally compact, non-compact, totally disconnected second countable Hausdorff space that has no isolated points. This set arises, for example, from the usual (compact) Cantor set $\Can$ by removing one point. One can also realise it as the product $\N\times\Can$, as the local field $\Q_p$, etc.

\begin{proof}[Proof of Theorem~\ref{thm:d:intro}] 
The ``only if'' part follows from Proposition~\ref{prop:lch} (and Lemma~\ref{lm:3conditions}). The conclusions about the crossed product $C_0(\Can^*) \rtimes_\red \Gamma$ follow from Proposition~\ref{prop:pi} and  Lemma~\ref{lm:3conditions}, respectively, from Proposition~\ref{prop:pi-actions}.

Suppose that $\Gamma$ is non-supramenable and let $A$ be a non-empty paradoxical subset of $\Gamma$. If $\Gamma$ is non-amenable, then choose $A$ such that $A \ll \Gamma$, cf.\ Lemma~\ref{lm:smallA}. Let $X_A$ and $\cC_A$ be as defined in the beginning of this section, and let $\cA$ be as in Proposition~\ref{prop:c}.

The \Cs{} $\cA$ contains the projection $1_A$, and this projection is furthermore $\Gamma$-full in $\cA$ by Proposition~\ref{prop:c}. It follows that $\cA$ contains a maximal proper closed $\Gamma$-invariant ideal $\cI$  (cf.\ Remark~\ref{rm:compact-ideals}). Put $\cB = \cA/\cI$. Let 
$$X = \widehat{\cB} \qquad \text{and} \qquad Y = \widehat{\cA}.$$
Then $\cB \cong C_0(X)$, $\cA \cong C_0(Y)$, and $X$ is a closed $\Gamma$-invariant subset of $Y$. We know from Proposition~\ref{prop:c} that  $\Gamma$ acts freely on $Y$, and the action is moreover amenable if $\Gamma$ is exact. These properties pass to the subset $X$, so the action of $\Gamma$ on $X$ is free, amenable (if $\Gamma$ is exact), and minimal (by maximality of $\cI$). 

We know from Proposition~\ref{prop:c} that $\cA$ is generated by its projections. Hence $\cB$ is also generated by its projections. Hence $X$ and $Y$ are totally disconnected. As $\cA$ and $\cB$ are separable we conclude that $X$ and $Y$ are second countable. We show below that the action of $\Gamma$ on $X$ is purely infinite. This clearly will imply that $X$ has no isolated points (an isolated point is compact-open and of course never paradoxical). If $\Gamma$ is non-amenable, then $A \ll \Gamma$ which by Lemma~\ref{lm:proper-paradoxical} entails that $\cB$ is non-unital, whence $X$ is non-compact. We arrive at the same conclusion (that $X$ is non-compact) if $\Gamma$ is amenable, because no action of an amenable group on a compact Hausdorff space can be purely infinite. We can therefore conclude that $X$ is homeomorphic to $\Can^*$. 

We finally show that the action of $\Gamma$ on $X$ is purely infinite. Let $K$ be a non-empty compact-open subset of $X$. Then $K$ is $\Gamma$-full in $X$ because $X$ is a minimal $\Gamma$-space. Hence there exists a $\Gamma$-full compact-open subset $K'$ of $Y$ such that $K= K' \cap X$, cf.\ Lemma~\ref{lm:extending}(i). Let $p \in \cA$ be the projection that corresponds to the projection $1_{K'} \in C_0(Y)$. Then $p$ is $\Gamma$-full in $\cA$ and therefore also $\Gamma$-full in $\cC_A$. Hence $p$ is $(\ell^\infty(\Gamma),\Gamma)$-paradoxical because $A$ is paradoxical, cf.\ Lemma~\ref{lm:paradox-proj}. By the construction of $\cA$, cf.\ Proposition~\ref{prop:c}, it follows that $p$ is $(\cA,\Gamma)$-paradoxical. Hence $1_{K'}$ is $(C_0(Y),\Gamma)$-paradoxical, so $K'$ is $(Y,\Gamma,\K)$-paradoxical, cf.\ claim (c) (below Definition~\ref{def:pi-action}). By Lemma~\ref{lm:extending}(ii) this implies that $K$ is $(X,\Gamma,\K)$-paradoxical.
\end{proof}

\section{Amenable actions of non-exact groups on locally compact Hausdorff spaces} 

\noindent It is well-known that only exact (discrete) group can act amenably on a compact Hausdorff space, cf.~\cite[Theorem 5.1.7]{BroOza:book}. We also know that any group (exact or not) admits an amenable action on some locally compact Hausdorff space, for example on the group itself. We shall here address the issue of when a non-exact group admits amenable actions on the locally compact non-compact Cantor set $\Can^*$, and whether the second part of Theorem~\ref{thm:d:intro} can be extended to all countable non-supramenable groups. The goal is to prove Theorem~\ref{thm:non-exact:intro}.

\medskip
We first recall a notion of induced action. Let $\Gamma_0$ be a subgroup of a countable group $\Gamma$, and suppose that $\Gamma_0$ acts on a locally compact Hausdorff space $X$. Then $\Gamma$ acts by left-multiplication on the first coordinate of $\Gamma \times X$, and $\Gamma_0$ acts on $\Gamma \times X$ by $s.(t,x) = (ts^{-1},s.x)$, for $s \in \Gamma_0$, $(t,x) \in \Gamma \times X$. These two actions commute and thus $\Gamma$ acts on $Y = (\Gamma \times X)/\Gamma_0$. This action is called the \emph{induced} action. We record some facts about this construction:

\begin{lemma}\label{lem:induced}
Let $\Gamma_0 \curvearrowright X$ and $\Gamma \curvearrowright Y$ be as above. Then:
\begin{enumerate}
\item $Y$ is homeomorphic to $X \times (\Gamma/\Gamma_0)$.
\item If $X\cong \Can$, then $Y\cong \Can$ or $Y\cong \Can^*$ according to whether $|\Gamma : \Gamma_0|$ is finite or infinite.
\item If $X\cong \Can^*$, then $Y\cong\Can^*$.
\item If $\Gamma_0 \curvearrowright X$ is free, then so is $\Gamma \curvearrowright Y$.
\item If $\Gamma_0 \curvearrowright X$ is minimal, then so is $\Gamma \curvearrowright Y$.
\item If $\Gamma_0 \curvearrowright X$ is amenable, then so is $\Gamma \curvearrowright Y$.
\item If $X$ admits a non-zero $\Gamma_0$-invariant Radon measure, then $Y$ admits a non-zero $\Gamma$-invariant Radon measure.
\item If $\Gamma_0 \curvearrowright X$ is purely infinite, then so is $\Gamma \curvearrowright Y$.
\end{enumerate}
\end{lemma}

\begin{proof}
Let $\pi \colon \Gamma \times X \to Y$ denote the quotient mapping.  

(i). Write $\Gamma = \bigcup_{\alpha \in I} t_\alpha \, \Gamma_0$ as a disjoint union of left cosets and put $A = \{t_\alpha \mid \alpha \in I\}$. Then $A \times X$ is a clopen subset of $\Gamma \times X$ which is a transversal for the action of $\Gamma_0$ on $\Gamma \times X$. The restriction of $\pi$ to $A \times X$ therefore defines a homeomorphism $A \times X \to Y$.

(ii) and (iii) follow from (i). 

(iv). Suppose that $t.\pi(s,x) = \pi(s,x)$ for some $t \in \Gamma$ and some $(s,x) \in \Gamma \times X$. Then there is $r \in \Gamma_0$ such that $(ts,x) = (sr^{-1},r.x)$. As $\Gamma_0$ acts freely on $X$ this implies that $r=e$, so $t=e$.

(v). Each $\Gamma \times \Gamma_0$ orbit on $\Gamma \times X$ is dense, if $\Gamma_0$ acts minimally on $X$. Hence each $\Gamma$ orbit on $Y$ is dense. 

(vi).  Let $m_i \colon X \to \mathrm{Prob}(\Gamma_0)$, $i \in I$, be approximate invariant continuous means that witness the amenability of the action of $\Gamma_0$ on $X$, cf.~\cite{Ana:TAMS2002}. For each $m \in \mathrm{Prob}(\Gamma_0)$ and for each $t \in \Gamma$ define $m_t \in  \mathrm{Prob}(\Gamma)$ by $m_t(E) = m(t^{-1}E \cap \Gamma_0)$, for $E \subseteq \Gamma$. Identify $Y$ with $A \times X$ as in (i), and define 
$$\widetilde{m}_i \colon A \times X \to  \mathrm{Prob}(\Gamma) \quad \text{by} \quad \widetilde{m}_i^{(t_\alpha,x)} = (m_i^x)_{t_\alpha}, \quad x \in X, \; \alpha \in I.$$ 
Let $s \in \Gamma$ and $(t_\alpha,x) \in A \times X$ be given. Then $st_\alpha = t_\beta \, r$ for some (unique) $\beta \in I$ and $r \in \Gamma_0$, and $s.(t_\alpha,x) = (t_\beta,r.x)$. Let $E \subseteq \Gamma$. Then
\begin{eqnarray*}
\widetilde{m}_i^{s.(t_\alpha,x)}(E) & =& \widetilde{m}_i^{(t_\beta,r.x)}(E) \; = \; m_i^{r.x}\big(t_\beta^{-1}E \cap \Gamma_0\big) \\ & = & m_i^{r.x}\big(rt_\alpha^{-1}s^{-1}E \cap \Gamma_0\big). 
\end{eqnarray*}
The latter expression is close to $m_i^{x}\big(t_\alpha^{-1}s^{-1}E \cap \Gamma_0\big)$ when $i$ is large by approximate invariance of $(m_i)$, and
$$s.\widetilde{m}_i^{(t_\alpha,x)}(E) \; = \; \widetilde{m}_i^{(t_\alpha,x)}(s^{-1}E) \;=\; m_i^x\big(t_\alpha^{-1}s^{-1}E \cap \Gamma_0\big).$$
This shows that $(\widetilde{m}_i)_i$ is an approximate invariant continuous mean.

 (vii). Identify again $Y$ with $A \times X$ as in (i). Let $\lambda$ be a non-zero $\Gamma_0$-invariant Radon measure on $X$ and let $\mu$ be counting measure on $A$. Then $\mu \otimes \lambda$ is a $\Gamma$-invariant non-zero Radon measure on $A \times X$.

(viii). If $F$ and $F'$ are disjoint compact-open subsets of $Y$ both of which are $(Y,\Gamma)$-paradoxical, then $F \cup F'$ is also $(Y,\Gamma)$-paradoxical. We need therefore only show that $F=\pi(\{t\} \times E)$ is $(Y,\Gamma)$-paradoxical for each $t \in \Gamma$ and for each  compact-open subset $E$ of $X$. As $E$ is $(X,\Gamma_0)$-paradoxical there are pairwise disjoint compact-open subsets $E_1, E_2, \dots, E_{n+m}$ of $X$ and elements $t_1,t_2, \dots, t_{n+m} \in \Gamma_0$ such that \eqref{eq:E} holds. Put $F_j = \pi(\{t\} \times E_j)$ and let $s_j \in \Gamma$ be such that $s_jt = tt_j$. Then $F_1,F_2, \dots, F_{n+m}$ are pairwise disjoint compact-open subsets of $F$. As 
$$s_j .\pi(t,x) = \pi(s_jt,x) = \pi(tt_j, x) = \pi(t,t_j.x),$$
we see that $s_j.F_j = \pi(\{t\} \times t_j.E_j)$. This shows that
$$F = \bigcup_{j=1}^n s_j.F_j = \bigcup_{j=n+1}^{n+m} s_j.F_j,$$
and hence that $F$ is $(Y,\Gamma)$-paradoxical.
\end{proof}

\begin{proof}[Proof of Theorem~\ref{thm:non-exact:intro}]
In each case we have a subgroup $\Gamma_0$ of $\Gamma$ which admits a free minimal amenable action on $X$, where $X = \Can$ or $X = \Can^*$. This action induces a free minimal amenable action of $\Gamma$ on $Y$, where $Y = \Can$ or $Y = \Can^*$ (according to the conclusions of~(ii) and~(iii) in Lemma~\ref{lem:induced}).

(i). Here $\Gamma_0$ acts on $X=\Can$ in the prescribed way by~\cite{HjorthMolberg}  or by~\cite{RorSie:ETDS}. The former reference gives a free minimal action on $\Can$ for every infinite $\Gamma_0$, and this action will be amenable if $\Gamma_0$ is amenable. The latter reference gives a free minimal amenable action on $\Can$, whenever $\Gamma_0$ is exact and non-amenable. We conclude by inducing the action and recall that a non-exact group cannot act amenably on a compact set. If $|\Gamma:\Gamma_0| = \infty$, then $Y = \Can^*$, see Lemma~\ref{lem:induced}.

(ii). It is known that $\Gamma_0 = \Z$ has many interesting minimal actions on $X=\Can^*$, see~\cite{Danilenko01}. These actions are necessarily free since every non-trivial subgroup of $\Z$ has finite index. Pick such an action; it is of course amenable since $\Z$ is amenable. Therefore it induces the desired action of $\Gamma$ on $Y=\Can^*$. If $\Gamma_0$ is infinite and amenable, then $\Gamma_0$ has a free minimal action on $X=\Can$ by~\cite{HjorthMolberg}, and this action is necessarily amenable. Again, if $|\Gamma:\Gamma_0| = \infty$, then $Y = \Can^*$.

(iii). Here $\Gamma_0$ is exact and non-supramenable. By Theorem~\ref{thm:d:intro} there is a free minimal amenable purely infinite action of $\Gamma_0$ on $X=\Can^*$. This action induces a free minimal amenable purely infinite action of $\Gamma$ on $Y = \Can^*$.
\end{proof}

\begin{questions}
Let $\Gamma$ be a countably infinite group.
\begin{enumerate}\setlength{\itemsep}{7pt}%
\item Does $\Gamma$ admit a free minimal amenable action on $\Can^*$? 

In that case, $C_0(\Can^*) \rtimes_\red \Gamma$ is a (non-unital) simple nuclear separable \Cs{} in the UCT class. 
\item Does $\Gamma$ admit a free minimal amenable action on $\Can^*$ which leaves invariant a non-zero Radon measure on $\Can^*$? 

In that case, $C_0(\Can^*) \rtimes_\red \Gamma$ is a (non-unital) simple nuclear separable \Cs{} in the UCT class with a densely defined trace (hence it is stably finite). 
\item Does $\Gamma$ admit a free minimal amenable purely infinite action on $\Can^*$ if $\Gamma$ is non-supramenable?

In that case, $C_0(\Can^*) \rtimes_\red \Gamma$ is a stable Kirchberg algebra in the UCT class.
\end{enumerate}
\end{questions}

\noindent It follows from Theorem~\ref{thm:non-exact:intro} that Question~(i) has an affirmative answer if $\Gamma$ contains an element of infinite order, or if $\Gamma$ contains an infinite exact subgroup of infinite index, or if $\Gamma$ contains an infinite exact subgroup and $\Gamma$ itself is non-exact. Question~(ii) has an affirmative answer if $\Gamma$ contains an element of infinite order, or if $\Gamma$ contains an infinite amenable subgroup of infinite index. Question~(iii) has an affirmative answer if $\Gamma$ contains an exact non-supramenable subgroup. 

On the other hand, as for Question (i), we do not even know if every countable \emph{amenable} group $\Gamma$ admits a free minimal action on $\Can^*$. If there is an (infinite) set $A \subseteq \Gamma$ such that $A \ll^* \Gamma$
and such that some minimal $\Gamma$-subspace of $X_A$ is non-discrete, then the construction in Proposition~\ref{prop:X_A} and Section~\ref{sec:Kirchberg} will give such an action. 
Example~\ref{ex:discrete} shows that the existence of a non-discrete minimal $\Gamma$-subspace of $X_A$ is not guaranteed, even if $A$ is infinite.\footnote{Added in proof: It will be shown in a subsequent paper by the third named author and H.\ Matui that each infinite countable group $\Gamma$ indeed does contain a subset $A$ with these properties.}

\bibliographystyle{amsplain}
\providecommand{\bysame}{\leavevmode\hbox to3em{\hrulefill}\thinspace}
\providecommand{\MR}{\relax\ifhmode\unskip\space\fi MR }
\providecommand{\MRhref}[2]{%
  \href{http://www.ams.org/mathscinet-getitem?mr=#1}{#2}
}
\providecommand{\href}[2]{#2}

\Addresses
\end{document}